\documentclass{amcjou}

\usepackage{amssymb}

\makeatletter
\let\th@plain=\undefined
\makeatletter

\usepackage[pdftex]{hyperref} 
\usepackage[amsmath,amsthm,thmmarks,hyperref]{ntheorem}
\usepackage[capitalize]{cleveref}

\let\oldbackslash=\backslash
\renewcommand{\backslash}{\oldbackslash\mkern-3mu}

\usepackage{graphicx}

\usepackage{bm}

\newcommand\bigset[2]{\left\{\, #1
 \mathrel{\left| \vphantom {\left\{ #1 \mid #2 \right\} }
 \right.} #2 \,\right\} }

\makeatletter
\newcommand{\oldsection}{}
\let\oldsection=\section
\newcommand{\newsection}[1]{\oldsection{\mathversion{bold}#1}}
\renewcommand{\section}{\@ifnextchar*{\oldsection}{\@ifnextchar[{\oldsection}{\newsection}}}
\makeatother

\makeatletter
\newcommand{\mytheoremheaderfont}[1]{{\normalfont\bfseries{#1}}}
\newcommand{\optargfont}[1]{{\normalfont #1}}

\renewtheoremstyle{plain}%
{\medbreak\item[\hskip\labelsep \mytheoremheaderfont{##1 ##2\@addpunct{.}}]}%
{\medbreak\item[\hskip\labelsep \mytheoremheaderfont{##1 ##2 }\optargfont{(##3)\@addpunct{.}}]}
\makeatother

\theoremstyle{plain}

\newtheorem{thm}[equation]{Theorem}
\newtheorem{cor}[equation]{Corollary}
\newtheorem{prop}[equation]{Proposition}
\newtheorem{lem}[equation]{Lemma}

\crefname{lem}{Lemma}{Lemmas}
\crefname{cor}{Corollary}{Corollaries}

\makeatletter
\global\theorembodyfont{\normalfont} 
\renewtheoremstyle{definition}%
{\medbreak\item[\hskip\labelsep \mytheoremheaderfont{##1 ##2\@addpunct{.}}]}%
{\medbreak\item[\hskip\labelsep \mytheoremheaderfont{##1 ##2 }\optargfont{(##3)\@addpunct{.}}]}
\makeatother

\theoremstyle{definition}
\newtheorem{rem}[equation]{Remark}
\newtheorem{notation}[equation]{Notation}
\newtheorem{defn}[equation]{Definition}

 \newtheorem{ack}{Acknowledgments\ignorespaces}[section]
 
\newcommand{\augment}{&}

\newcommand{\integer}{\mathbb{Z}}
\newcommand{\iso}{\cong}
 \newcommand{\Cay}{\mathop{\rm Cay}}
 \newcommand{\bad}{\textbf{{?}{?}{?}}}
\renewcommand{\pmod}[1]{\ (\mathop{\rm mod} #1)}

\newcommand{\normal}{\triangleleft}
\newcommand{\notnormal}{\mathrel{\not\hskip-2pt\normal}}
\DeclareMathOperator{\Aut}{Aut}

\newcommand{\pref}[1]{(\ref{#1})}
\newcommand{\fullcref}[2]{\cref{#1}\pref{#1-#2}}
\newcommand{\see}[1]{{\rm(}see~\ref{#1}{\rm)}}

\numberwithin{equation}{section}

 \newcounter{case}
 \newenvironment{case}[1][\unskip]{\refstepcounter{case}\bf
 \medskip \noindent Case \thecase\ #1. \it}{\unskip\upshape}
 \renewcommand{\thecase}{\arabic{case}}
\crefformat{case}{Case~#2#1#3}

 \newcounter{subcase}
 \newenvironment{subcase}[1][\unskip]{\refstepcounter{subcase}\bf
 \medskip \noindent \hskip\parindent Subcase \thesubcase\ #1. \it}{\unskip\upshape}
\numberwithin{subcase}{case}
\crefformat{subcase}{Subcase~#2#1#3}
\crefname{subcase}{Subcase}{Subcases}

 \newcounter{subsubcase}
 \newenvironment{subsubcase}[1][\unskip]{\refstepcounter{subsubcase}\bf
 \medskip \noindent \hskip2\parindent Subsubcase \thesubsubcase #1. \it}{\unskip\upshape}
\numberwithin{subsubcase}{subcase}
\crefformat{subsubcase}{Subsubcase~#2#1#3}

 \newcounter{subsubsubcase}

\newenvironment{subsubsubcase}[1][\unskip]{\refstepcounter{subsubsubcase}\bf
 \medskip \noindent \hskip3\parindent Subsubsubcase \thesubsubsubcase #1. \it
}{\unskip\upshape} \numberwithin{subsubsubcase}{subsubcase}
\crefformat{subsubsubcase}{Subsubsubcase~#2#1#3}
\crefname{subsubsubcase}{Subsubsubcase}{Subsubsubcases}

\newcounter{saveBibCtr}

\begin{document}

\begin{frontmatter}   

\titledata{Hamiltonian cycles in Cayley graphs whose order has few prime factors}	  
{}	        

\authordata{K.\,Kutnar}
{University of Primorska, FAMNIT, Glagolja\v ska 8, 6000 Koper, Slovenia} 
{Klavdija.Kutnar@upr.si}		   
{}	 

\authordata{D.\,Maru\v si\v c}	   
{University of Primorska, FAMNIT, Glagolja\v ska 8, 6000 Koper, Slovenia
\\ University of Ljubljana, PEF, Kardeljeva pl. 16, 1000 Ljubljana, Slovenia} 
{Dragan.Marusic@guest.arnes.si}
{}

\authordata{D.\,W.\,Morris}	   
{Department of Mathematics and Computer Science,
University of Lethbridge,
Lethbridge, Alberta, T1K~3M4, Canada}
{Dave.Morris@uleth.ca}
{}

\authordata{J.\,Morris}      
{Department of Mathematics and Computer Science,
University of Lethbridge,
Lethbridge, Alberta, T1K~3M4, Canada}
{Joy.Morris@uleth.ca}
{}

\authordata{P.\,\v Sparl}	   
{University of Ljubljana, PEF, Kardeljeva pl. 16, 1000 Ljubljana, Slovenia}
{Primoz.Sparl@pef.uni-lj.si}
{}


\begin{abstract}
We prove that if $\Cay(G;S)$ is a connected Cayley graph with $n$ vertices, and the prime factorization of~$n$ is very small, then $\Cay(G;S)$ has a hamiltonian cycle. More precisely, if $p$, $q$, and~$r$ are distinct primes, then $n$ can be of the form $kp$ with $24 \neq k < 32$, or  of the form $kpq$ with $k \le 5$, or of the form $pqr$, or of the form $kp^2$ with $k \le 4$, or of the form $kp^3$ with $k \le 2$.
\end{abstract}

\end{frontmatter}   

\section{Introduction} \label{IntroSect}

\begin{defn}
Let $S$ be a subset of a finite group~$G$. The \emph{Cayley graph} $\Cay(G;S)$ is the graph whose vertices are the elements of~$G$, with an edge joining $g$ and~$gs$, for every $g \in G$ and $s \in S$.
\end{defn}

It was conjectured in the early 1970's that every connected Cayley graph has a hamiltonian cycle, but we are still nowhere near a resolution of this problem.
(See the surveys \cite{CurranGallian-survey,PakRadoicic-survey,WitteGallian-survey} for discussions of the progress that has been made.)
One of the purposes of this paper is to provide some evidence for the conjecture, by establishing that all Cayley graphs on groups of small order have hamiltonian cycles.
Our results are summarized in the following theorem:

\begin{thm} \label{Ham<100}
 Let $G$ be a finite group. Every connected Cayley graph on~$G$ has a
hamiltonian cycle if\/ $|G|$ has any of the following forms\/ {\upshape(}where
$p$, $q$, and~$r$ are distinct primes\/{\upshape):}
 \begin{enumerate}
 
 \item \label{Ham<100-kp}
 \newcounter{kp} \setcounter{kp}{\theenumi}
 $k p$, where $1 \le k < 32$, with $k \neq 24$, 
 
 \item \label{Ham<100-kpq}
 \newcounter{kpq} \setcounter{kpq}{\theenumi}
 $k p q$, where $1 \le k \le 5$,
 
 \item \label{Ham<100-pqr}
 \newcounter{pqr} \setcounter{pqr}{\theenumi}
 $p q r$,
 
 \item \label{Ham<100-kp2}
 \newcounter{kpsquared} \setcounter{kpsquared}{\theenumi}
 $k p^2$, where $1 \le k \le 4$, 
 
 \item \label{Ham<100-kp3}
 \newcounter{kpcubed} \setcounter{kpcubed}{\theenumi}
 $k p^3$, where $1 \le k \le 2$.
 

 \end{enumerate}
 \end{thm}

 \begin{rem}[\cite{Witte-pn}] \label{pk}
 It is also known that Cayley graphs with $p^k$ vertices all have hamiltonian cycles.
 \end{rem}

 This work began in the 1980's as an undergraduate research project by D.\,Jungreis and E.\,Friedman at the University of Minnesota, Duluth, under the supervision of J.\,A.\,Gallian, but their results \cite{JungreisFriedman} were never published. (This paper is a revision and extension of the students' work; we include statements and proofs of their main results.)  
 We consider only Cayley graphs in this paper; see \cite{KutnarSparl} for references to analogous work on hamiltonian cycles in more general vertex-transitive graphs with a small number of vertices.

It was originally expected that the numerous available methods would easily prove that every Cayley graph on any group of order less than, say, $100$ has a hamiltonian cycle. Unfortunately, a major lesson of this work is that such an expectation is wildly incorrect.
Namely, although the results here were not obtained easily, they do not even include all of the orders up to $75$. More precisely, as can be seen from \cref{To120}, combining \cref{Ham<100} with \cref{pk} deals with all orders less than $120$, \emph{except:}
 \begin{itemize}
 \item $72 = 2^3 \cdot 3^2 = 8 p^2$ or $24p$,
 \item $96 = 2^5 \cdot 3 = 32 p$,
 \item $108 = 2^2 \cdot 3^3 = 36p$ or $4p^3$,
 \item $120 = 2^3 \cdot 3 \cdot 5 = 24 p$.
 \end{itemize}
In fact, the situation is even worse than this list would seem to indicate, because the cases $k = 16$, $k = 27$, and $k = 30$ of \fullcref{Ham<100}{kp} are not proved here: they were treated in the separate papers \cite{CurranMorris2-16p,GhaderpourMorris-27p,GhaderpourMorris-30p} after a preprint of this paper was released. 

\begin{figure}[ht]
\begin{align*}
 1 & \augment 
 21 &= 3p \augment
 41 &= p \augment
 61 &= p \augment
 81 &= p^k \augment
 101 &= p \augment
 \\
 2 &= p \augment
 22 &= 2p \augment
 42 &= 6p \augment
 62 &= 2p \augment
 82 &= 2p \augment
 102 &= 6p \augment
 \\
 3 &= p \augment
 23 &= p \augment
 43 &= p \augment
 63 &= 9p  \augment
 83 &= p \augment
 103 &= p \augment
 \\
 4 &= p^k \augment
 24 &= 8p \augment
 44 &= 4p \augment
 64 &= p^k \augment
 84 &= 12p \augment
 104 &= 8p \augment
 \\
 5 &= p \augment
 25 &= p^2 \augment
 45 &= 9p \augment
 65 &= 5p \augment
 85 &= 5p \augment
 105 &= 15p \augment
 \\
 6 &= 2p \augment
 26 &= 2p \augment
 46 &= 2p \augment
 66 &= 6p \augment
 86 &= 2p \augment
 106 &= 2p \augment
 \\
 7 &= p \augment
 27 &= p^k \augment
 47 &= p \augment
 67 &= p \augment
 87 &= 3p \augment
 107 &= p \augment
 \\
 8 &= p^k \augment
 28 &= 4p \augment
 48 &= 16p \augment
 68 &= 4p \augment
 88 &= 8p \augment
 108 &= \bad \augment
 \\
 9 &= p^k \augment
 29 &= p \augment
 49 &= p^k \augment
 69 &= 3p \augment
 89 &= p \augment
 109 &= p \augment
 \\
 10 &= 2p \augment
 30 &= 6p \augment
 50 &= 2p^2 \augment
 70 &= 2pq \augment
 90 &= 18p \augment
 110 &= 10p \augment
 \\
 11 &= p \augment
 31 &= p \augment
 51 &= 3p \augment
 71 &= p \augment
 91 &= pq \augment
  111 &= 3p \augment
\\
 12 &= 4p \augment
 32 &= p^k \augment
 52 &= 4p \augment
 72 &= \bad \augment
 92 &= 4p \augment
  112 &= 16p \augment
\\
 13 &= p \augment
 33 &= 3p \augment
 53 &= p \augment
 73 &= p \augment
 93 &= 3p \augment
 113 &= p \augment
\\
 14 &= 2p \augment
 34 &= 2p \augment
 54 &= 2p^3 \augment
 74 &= 2p \augment
 94 &= 2p \augment
 114 &= 6p \augment
 \\
 15 &= 3p \augment
 35 &= pq \augment
 55 &= pq \augment
 75 &=  3p^2 \augment
 95 &= 5p \augment
 115 &= 5p \augment
 \\
 16 &= p^k \augment
 36 &= 4p^2 \augment
 56 &= 8p \augment
 76 &= 4p \augment
 96 &= \bad \augment
 116 &= 4p \augment
 \\
 17 &= p \augment
 37 &= p \augment
 57 &= 3p \augment
 77 &= pq \augment
 97 &= p \augment 
 117 &= 9p \augment
 \\
 18 &= 2p^2 \augment
 38 &= 2p \augment
 58 &= 2p \augment
 78 &= 6p \augment
 98 &= 2p^2 \augment
 118 &= 2p \augment
 \\
 19 &= p \augment
 39 &= 3p \augment
 59 &= p \augment
 79 &= p \augment
 99 &= 9p \augment
 119 &= 7p \augment
 \\
 20 &= 4p \augment
 40 &= 8p \augment
 60 &= 12p \augment
 80 &= 16p \augment
 100 &= 4p^2 \augment
 120 &= \bad \augment
 \end{align*}
 \caption{Factorizations of orders up to 120.}
 \label{To120}
 \end{figure}

 \subsection*{Outline of the paper.}
 Most of the cases of \cref{Ham<100} are known (including all of the cases where $k = 1$). For example, C.\,C.\,Chen and N.\,Quimpo \cite{ChenQuimpo-pq} proved that Cayley graphs of order $pq$ are hamiltonian (in fact, edge-hamiltonian), and D.\,Li \cite{Li-pqr} proved that Cayley graphs of order $pqr$ are hamiltonian. (However, the latter result is in Chinese, so we provide a proof.) The following list of the paper's sections enumerates the main cases that need to be considered.
\begin{align*} \itemsep=-2pt
&\text{\S\ref{PrelimSect} {Preliminaries}}
&&\text{\S\ref{4p2Sect} {Groups of order $4p^2$}}
&&\text{\S\ref{2p3Sect} {Groups of order $2p^3$} }
\\
&\text{\S\ref{8p-section} {Groups of order $8p$}}
&&\text{\S\ref{pqrSect} {Groups of order $pqr$} }
&&\text{\S\ref{18pSect} {Groups of order $18p$}}
\\
&\text{\S\ref{3p2Sect} {Groups of order $3p^2$}}
&&\text{\S\ref{4pqSect} {Groups of order $4pq$}}
\end{align*}

\section{Preliminaries} \label{PrelimSect}

\subsection{Outline of the proof of \cref{Ham<100}:}
Here is a description of how the results of this paper combine to prove \cref{Ham<100}.

\begin{enumerate} \renewcommand{\labelenumi}{(\theenumi)}

\setcounter{enumi}{\thekp}\addtocounter{enumi}{-1}
\item
If $k \in \{2, 3, 5, 6, 7, 10, 11, 13, 14, 17, 19, 22, 23, 26, 29, 31\}$, then $k$ is either prime or twice a prime, so $kp$ is of the form $pq$, $2pq$, $p^2$, or $2p^2$. These cases are treated below, in 
\pref{HowToProve-kpq-1}, \pref{HowToProve-kpq-2}, 
\pref{HowToProve-kp2-1}, and~\pref{HowToProve-kp2-2},
respectively,
so we need only consider the other values of~$k$.
Also, we note that the proofs of 
\pref{Ham<100-kpq}--\pref{Ham<100-kp3} 
make no use of \pref{Ham<100-kp}, other than the cases $4p$ and $8p$, so we are free to employ any and all other parts of the theorem in establishing the cases of~\pref{Ham<100-kp} (other than $4p$ and $8p$).

 \begin{enumerate}
 \item[$1p$:] Groups of prime order are abelian, so \cref{abel} applies.
 \item[$4p$:] See \cref{4p}. 
 \item[$8p$:] See \cref{8p}.
 \item[$9p$:] \Cref{p2q-easy} applies unless $p = 2$. If $p = 2$, then $|G|$ is of the form $2p^2$.
 \item[$12p$:] $|G|$ is of the form $8p$ (if $p = 2$) or $4p^2$ (if $p = 3$) or $4pq$ (if $p > 3$).
 \item[$15p$:] $|G|$ is of the form $3p^2$ (if $p = 5$) or $3pq$ (otherwise).
 \item[$16p$:] See \cite{CurranMorris2-16p}.
 \item[$18p$:] See \cref{18p}. 
 \item[$20p$:] $|G|$ is of the form $4p^2$ (if $p = 5$) or $4pq$ (otherwise).
 \item[$21p$:] $|G|$ is of the form $3p^2$ (if $p = 7$) or $3pq$ (otherwise).
 \item[$25p$:] \Cref{p2q-easy} applies unless $p \in \{2,3\}$. In the exceptional cases, $|G|$ is of the form $kp^2$ with $1 \le k \le 4$.
 \item[$27p$:] See \cite{GhaderpourMorris-27p}.
 \item[$28p$:] $|G|$ is of the form $4p^2$ (if $p = 7$) or $4pq$ (otherwise).
 \item[$30p$:] See \cite{GhaderpourMorris-30p}.
\end{enumerate}

\setcounter{enumi}{\thekpq}\addtocounter{enumi}{-1}
 \item
 Assume $|G| = kpq$ with $1 \le k \le 5$.
 	\begin{enumerate}
	\item \label{HowToProve-kpq-1}
	If $k = 1$, then $[G,G]$ is cyclic of prime order, so \cref{KeatingWitte} applies.
	\item  \label{HowToProve-kpq-2}
	If $k = 2$, see \cref{2pq}. 
	\item If $k = 3$, see \cref{3pq}. 
	\item If $k = 4$, see \cref{4pq}. 
	\item If $k = 5$, see \cref{5pq}. 
	\end{enumerate}

\setcounter{enumi}{\thepqr}\addtocounter{enumi}{-1}
 \item
 Assume $|G| = pqr$. See \cref{pqr} (or \cite{Li-pqr}). 

\setcounter{enumi}{\thekpsquared}\addtocounter{enumi}{-1}
\item
Assume $|G| = k p^2$ with $1 \le k \le 4$.
	\begin{enumerate}
	\item \label{HowToProve-kp2-1}
	 If $k = 1$, then $|G| = p^2$, so $G$ is abelian. Hence, \cref{abel} applies.
	\item  \label{HowToProve-kp2-2}
	If $k = 2$, see \cref{2p2}. 
	\item  If $k = 3$, see \cref{3p2}. 
	\item If $k = 4$, see \cref{4p2}. 
	\end{enumerate}

\setcounter{enumi}{\thekpcubed}\addtocounter{enumi}{-1}
 \item
 Assume $|G| = kp^3$ with $1 \le k \le 2$. 
 	\begin{enumerate}
	\item If $k = 1$, then $|G| = p^3$ is a prime power, so \cref{pk} applies.
	\item If $k = 2$, see \cref{2p3}. 
	\end{enumerate}
 
%
 \end{enumerate}

\subsection{Some basic results on Cayley graphs of small order}

It is very easy to see that Cayley graphs on abelian groups are hamiltonian (in fact, they are edge-hamiltonian \cite{ChenQuimpo-pq} and are usually hamiltonian connected \cite{ChenQuimpo-hamconn}):

\begin{lem}[\cite{ChenQuimpo-hamconn}] \label{abel}
If $G$ is abelian, then every connected Cayley graph on~$G$ has a hamiltonian cycle.
\end{lem}

The following generalization handles many groups of small order:

\begin{thm}[Keating-Witte \cite{KeatingWitte}] \label{KeatingWitte}
 If the commutator subgroup~$[G,G]$ of~$G$ is a cyclic $p$-group, then
every connected Cayley graph on~$G$ has a hamiltonian cycle.
 \end{thm}
 
 For ease of reference, we record a well-known (and easy) consequence of this theorem.

\begin{cor} \label{p2q-easy}
 If\/ $|G| = p^2 q$, where $p$ and~$q$ are primes with $p^2 \not\equiv 1 \pmod{q}$, then every
connected Cayley graph on~$G$ has a hamiltonian cycle.
 \end{cor}

\begin{proof}
We may assume $p \neq q$, for otherwise $|G| = p^3$ is a prime power, so \cref{pk} applies.

 Let $Q$ be a Sylow $q$-subgroup of~$G$. From Sylow's Theorem \pref{NoMod1->Normal}, we know that $Q$ is normal in~$G$. The quotient group $G/Q$,
being of order~$p^2$, must be abelian. Therefore $[G,G] \subset Q$ is cyclic of
order~$q$ or~$1$, so \cref{KeatingWitte} applies.
 \end{proof}

The proof of \cref{pk} actually yields the following stronger result:

\begin{cor}[{}{\cite[Cor.~3.3]{Morris-2genNilp}}] \label{pkSubgrp}
  Suppose
 \begin{itemize}
 \item $S$ is a generating set of~$G$,
 \item $N$ is a normal $p$-subgroup of~$G$,
 and
 \item $s t^{-1} \in N$, for all $s,t \in S$.
 \end{itemize}
 Then $\Cay(G;S)$ has a hamiltonian cycle.
 \end{cor}

\subsection{Factor Group Lemma}

When proving the various parts of \cref{Ham<100}, we will implicitly assume, by induction on~$|G|$, that if $N$ is any nontrivial, normal subgroup of~$G$, then every connected Cayley graph on $G/N$ has a hamiltonian cycle. (Similarly, we also assume that if $H$ is any proper subgroup of~$G$, then every connected Cayley graph on $H$ has a hamiltonian cycle.) Thus it is very useful to know when we can lift hamiltonian cycles from a quotient graph to the original Cayley graph. Here are a few well-known results of this type.

\begin{notation}
For $s_1,s_2,\ldots,s_n \in S \cup S^{-1}$, we use 
	$$ (s_1,s_2,s_3,\ldots,s_n) $$
to denote the walk in $\Cay(G;S)$ that visits (in order) the vertices
	$$ e, s_1, s_1s_2, s_1s_2s_3, \ldots, s_1s_2 \cdots s_n .$$
Also, 
	\begin{itemize}
	\item $(s_1,s_2,s_3,\ldots,s_n)^k$ denotes the walk that is obtained from the concatenation of $k$~copies of $(s_1,s_2,s_3,\ldots,s_n)$, 
	and 
	\item $(s_1,s_2,s_3,\ldots,s_n)\#$ denotes the walk  $(s_1,s_2,s_3,\ldots,s_{n-1})$ that is obtained by deleting the last term of the sequence.
	\end{itemize}
\end{notation}

The following observation is elementary.

\begin{lem} \label{FGL(notnormal)}
 Suppose
 \begin{itemize}
 \item $S$ is a generating set of~$G$,
 \item $H$ is a cyclic subgroup of~$G$, with index $|G:H| = n$,
 \item $s_1,s_2,\ldots,s_n$ is a sequence of $n$ elements of $S \cup S^{-1}$, such that 
 	\begin{itemize}
	\item the elements $e, s_1, s_1s_2, s_1s_2s_3, \ldots, s_1s_2\cdots s_{n-1}$ are all in different right cosets of~$H$,
	and
	\item the product $s_1s_2s_3\cdots s_n$ is a generator of~$H$.
	\end{itemize}
 \end{itemize}
 Then $(s_1,\ldots,s_n)^{|H|}$ is a hamiltonian cycle in $\Cay(G;S)$.
 \end{lem}
 
 The assumptions on the sequence $s_1,s_2,\ldots,s_n$ can also be expressed by saying that a certain quotient multigraph has a hamiltonian cycle:
 
 \begin{defn}
 If $H$ is any subgroup of~$G$, then $H \backslash \Cay(G;S)$ denotes the multigraph in which:
 	\begin{itemize}
	\item the vertices are the right cosets of~$H$,
	and
	\item there is an edge joining $Hg_1$ and~$Hg_2$ for each $s \in S \cup S^{-1}$, such that $g_1 s \in H g_2$.
	\end{itemize}
Thus, if there are two different elements $s_1$ and~$s_2$ of $S \cup S^{-1}$, such that $g_1 s_1$ and $g_1 s_2$ are both in $H g_2$, then the vertices $Hg_1$ and $Hg_2$ are joined by a double edge.
 \end{defn}

When the cyclic subgroup~$H$ is normal, we have the following well-known special case:

\begin{cor}[{``Factor Group Lemma''}] \label{FGL}
 Suppose
 \begin{itemize}
 \item $S$ is a generating set of~$G$,
 \item $N$ is a cyclic, normal subgroup of~$G$,
 \item $(s_1 N,\ldots,s_n N)$ is a hamiltonian cycle in $\Cay(G/N;S)$,
 and
 \item the product $s_1 s_2 \cdots s_n$ generates~$N$.
 \end{itemize}
 Then $(s_1,\ldots,s_n)^{|N|}$ is a hamiltonian cycle in $\Cay(G;S)$.
 \end{cor}

When $|H|$ (or $|N|$) is prime, it is generated by any of its nontrivial elements.
So, in order to know that there is a hamiltonian cycle for which the product $s_1 s_2 \cdots s_n$ generates~$H$, it suffices to know that there are two hamiltonian cycles that differ in only one edge:

\begin{cor} \label{MultiDouble}
  Suppose
 \begin{itemize}
 \item $S$ is a generating set of~$G$,
 \item $H$ is a subgroup of~$G$, such that $|H|$ is prime,
 \item the quotient multigraph $H \backslash \Cay(G;S)$ has a hamiltonian cycle~$C$,
 and
 \item $C$ uses some double edge of~$H \backslash \Cay(G;S)$.
 \end{itemize}
 Then there is a hamiltonian cycle in $\Cay(G;S)$.
 \end{cor}

\begin{defn}
We say that a generating set~$S$ of a group~$G$ is \emph{minimal} if no proper subset of~$S$ generates~$G$.
\end{defn}

\begin{cor} \label{DoubleEdge}
  Suppose
 \begin{itemize}
 \item $N$ is a normal subgroup of~$G$, such that $|N|$ is prime,
 \item the image of $S$ in $G/N$ is a minimal generating set of~$G/N$,
 \item there is a hamiltonian cycle in $\Cay(G/N;S)$,
 and
 \item $s \equiv t \pmod{N}$ for some $s,t \in S \cup S^{-1}$ with $s
\neq
t$.
 \end{itemize}
 Then there is a hamiltonian cycle in $\Cay(G;S)$.
 \end{cor}
 
 We will also use the following generalization of \cref{FGL(notnormal)}:
 
 \begin{lem}[{}{\cite[Lem.~5.1]{Witte-pn}}] \label{FGL(skewgen)}
 Suppose
 	\begin{itemize}
	\item $K$ is a normal subgroup of a subgroup~$H$ of~$G$,
	\item $(s_1,s_2,\ldots,s_n)$ is a hamiltonian cycle in the quotient $H \backslash \Cay(G;S)$,
	and
	\item the product $s_1s_2\cdots s_n$ generates $H/K$.
	\end{itemize}
Then $(s_1,s_2,\ldots,s_n)^{|H/K|}$ is a hamiltonian cycle in $K \backslash \Cay(G;S)$.
\end{lem}
 
 The theory of ``voltage graphs'' \cite[Thm.~2.1.3, p.~63]{GrossTucker} 
 (or see \cite[Thm.~5.2]{Alspach-lifting}) provides a method for applying \cref{FGL(notnormal)}. Here is one example that we will use:


 
\begin{thm}[Locke-Witte, c.f.\ {}{\cite[Prop.~3.3]{LockeWitte}}]\label{VoltageCor}
 Suppose 
 	\begin{itemize}
	\item $\Cay(G;S)$ is connected,
	\item $N$ is a normal subgroup of~$G$,
	\item $|N|$ is prime,
	and 
	\item for some~$k$, $\Cay(G/N; S)$ is isomorphic to either
		\begin{itemize}
		\item $\Cay \bigl( \integer_{4k}; \{1,2k\} \bigr)$ {\upshape(}a non-bipartite M\"obius ladder\/{\upshape)},
		or
		\item $\Cay \bigl( \integer_{2k} \times \integer_2; \{(1,0), (0,1)\} \bigr)$ {\upshape(}a bipartite prism\/{\upshape)}, with $2k \not\equiv 1 \pmod{|N|}$.
		\end{itemize}
	\end{itemize}
Then some hamiltonian cycle in $\Cay(G/N; S)$ lifts to a hamiltonian cycle in $\Cay(G;S)$.
\end{thm}

\subsection{Applications of \cref{FGL(notnormal)}}

For future reference, we record some special cases of \cref{FGL(notnormal)}. Although the hypotheses of these results are very restrictive (and rather complicated), they will be used many times.

\begin{lem}[Jungreis-Friedman {}{\cite[Lem.~6.1]{JungreisFriedman}}]\label{Stud61}
 Let $\{s_1,s_2\}$ generate the group $G$.  If
 \begin{itemize}
 \item $2 |s_1| \cdot | [s_1,s_2] | = |G|$,
 \item $s_2 \notin \langle s_1 \rangle \langle [s_1,s_2]
\rangle$,
 \item $\langle [s_1,s_2] \rangle \cap \langle s_1 \rangle
= \{e\}$, and
 \item $\langle [s_1,s_2] \rangle \cap \bigl( s_2^{-1}
\langle s_1 \rangle s_2 \bigr) = \{e\}$,
 \end{itemize}
 then  $\Cay \bigl( G;\{s_1,s_2\} \bigr)$ has a hamiltonian cycle.
 \end{lem}

\begin{proof}
 For convenience, let $\gamma = [s_1,s_2] = s_1^{-1}
s_2^{-1} s_1 s_2$. We claim that
 $$ ( s_1^{|s_1|-1}, s_2^{-1}, s_1^{-(|s_1|-1)},
s_2)^{|\gamma|} $$ is a hamiltonian cycle. 
This will follow from \cref{FGL(notnormal)} if we show that the vertices of the walk 
$ ( s_1^{|s_1|-1}, s_2^{-1}, s_1^{-(|s_1|-1)})$ 
are all in different right cosets of~$\langle \gamma \rangle$.

Note that the vertices in this walk are all in $\langle s_1 \rangle$ or $s_1^{-1} s_2^{-1} \langle s_1 \rangle$, and that 
	\begin{align} \label{Stud71pf-s1s2=s2}
	\langle \gamma \rangle s_1^{-1} s_2^{-1} \langle s_1 \rangle 
		= \langle \gamma \rangle (s_1^{-1} s_2^{-1} s_1 s_2) s_2^{-1} s_1^{-1} \langle s_1 \rangle 
		= \langle \gamma \rangle s_2^{-1} \langle s_1 \rangle 
		. \end{align}
	\begin{itemize}
	\item Since $\langle \gamma \rangle \cap \langle s_1 \rangle = \{e\}$, we know that all of the elements of $\langle s_1 \rangle$ are in different right cosets.
	\item Since $\langle \gamma \rangle \cap s_2^{-1} \langle s_1 \rangle s_2 = \{e\}$, we know that all of the elements of $s_2^{-1} \langle s_1 \rangle$ are in different right cosets. So \pref{Stud71pf-s1s2=s2} implies that all of the elements of $s_1^{-1} s_2^{-1} \langle s_1 \rangle$ are in different right cosets.
	\item
	Since $s_2^{-1} \notin \langle \gamma \rangle \langle s_1 \rangle$, we know that
		$$ \langle \gamma \rangle  \langle s_1 \rangle \cap 
		\langle \gamma \rangle s_2^{-1} \langle s_1 \rangle  
		= \emptyset .
		$$
	So \pref{Stud71pf-s1s2=s2}  implies that none of the elements of $\langle s_1 \rangle$ are in the same right coset as any element of $s_1^{-1} s_2^{-1} \langle s_1 \rangle$.
	\qedhere
	\end{itemize}
 \end{proof}

\Cref{Stud61} will be used many times; here is an example.

\begin{cor} \label{Z3kx(Z2xZ2)}
If $G \iso \integer_{3^k} \ltimes (\integer_2 \times \integer_2)$, for some $k \in \integer^+$, then every connected Cayley graph on~$G$ has a hamiltonian cycle.
\end{cor}

\begin{proof}
We may assume $\integer_{3^k}$ acts nontrivially on $\integer_2 \times \integer_2$, for otherwise $G$ is abelian, so \cref{KeatingWitte} applies. Then $\#S = 2$, and some element~$s_1$ of~$S$ generates (a conjugate of)~$\integer_{3^k}$. The other element~$s_2$ of~$S$ is of the form $s_1^i y$ with $y \in \integer_2 \times \integer_2$, so we have 
	$$[s_1,s_2] = [s_1,y] \in (\integer_2 \times \integer_2) - \langle y \rangle ,$$
so it is easy to verify the hypotheses of \cref{Stud61}.
\end{proof}

 \begin{cor} \label{4p}
 If\/ $|G| = 4p$, where $p$ is prime, then every connected Cayley graph on~$G$ has a hamiltonian cycle.
 \end{cor}
 
 \begin{proof}
 \Cref{p2q-easy} applies unless $p = 3$. However, if $p = 3$, then either the Sylow $3$-subgroup is normal, so the argument of \cref{p2q-easy} applies, or $G \iso A_4 \iso \integer_3 \ltimes (\integer_2 \times \integer_2)$, so \cref{Z3kx(Z2xZ2)} applies. 
\end{proof}

\begin{lem}[Jungreis-Friedman {}{\cite[Lem.~7.1]{JungreisFriedman}}] \label{Stud71}
 Let $S$ be a minimal generating set for the group~$G$.  If there exist
two distinct generators $s_1, s_2 \in S$ such that: \begin{itemize}
\item
$|s_1s_2|=|G|/|\langle S-\{s_1\}\rangle|$, \item $\langle
s_1s_2\rangle\cap \langle S-\{s_1\}\rangle=\{e\}$, and \item there is a
hamiltonian cycle in $\Cay \bigl( \langle S-\{s_1\}\rangle; S-\{s_1\} \bigr)$,
 \end{itemize}
 then there is a hamiltonian cycle in $\Cay(G;S)$.
 \end{lem}

\begin{proof}
 Let $(t_i)_{i=1}^{n}$ be a hamiltonian cycle in $\Cay \bigl( \langle S-\{s_1\}\rangle;
S-\{s_1\} \bigr)$.  Since $S$ is a minimal generating set for $G$, we know that
$s_2$ or its inverse must appear somewhere in this cycle, and by choosing a different
starting point if necessary, and reversing the cycle if necessary, we can assume without loss of generality
that $t_n=s_2^{-1}$.  Then $t_1t_2\ldots t_{n-1}=s_2$.

Since $\langle s_1s_2 \rangle \cap \langle S - \{s_1 \} \rangle = \{e\}$, conjugating by $s_2^{-1}$ tells us that $\langle s_2s_1 \rangle \cap \langle S - \{s_1 \} \rangle$ is also trivial.
So the elements of $\langle S - \{s_1 \} \rangle$ are all in different right cosets of $\langle s_2s_1 \rangle$. Therefore \cref{FGL(notnormal)} tells us that 
$$\bigl( (t_i)_{i=1}^{n-1} , s_1 \bigr)^{|s_1s_2|}$$ is a hamiltonian
cycle in
$\Cay(G;S)$. 
 \end{proof}

\begin{cor}\label{direct-prod-cor}
 Let $S$ be a minimal generating set for the group $G$.  If  there exist two distinct
generators $s_1,s_2\in S$, such that 
	\begin{itemize}
	\item $\Cay \bigl( \langle S-\{s_1\}\rangle;S-\{s_1\} \bigr)$ has a
hamiltonian cycle, 
	and 
	\item $|s_1s_2|=|G|/|\langle S-\{s_1\}\rangle|$ is prime,
	\end{itemize}
then there is a hamiltonian cycle in $\Cay(G;S)$.
 \end{cor}

\begin{proof}
 In order to apply \cref{Stud71}, we need only show that
 $\langle s_1s_2\rangle\cap \langle S-\{s_1\}\rangle=\{e\}$.
Suppose, to the contrary, that $\langle
S-\{s_1\}\rangle$ contains a nontrivial element of $\langle s_1s_2\rangle$.
Since $|s_1s_2|$ is prime, this implies that $s_1s_2 \in \langle S-\{s_1\}\rangle$.
But, since $s_1$ and $s_2$ are distinct, we also have $s_2 \in \langle S-\{s_1\}\rangle$.
Therefore $s_1 \in \langle S-\{s_1\}\rangle$, contradicting the minimality of~$S$.
 \end{proof}

 \subsection{Groups of dihedral type} 

\begin{notation} 
We use $D_{2n}$ and $Q_{4n}$ to denote the \emph{dihedral group} of order~$2n$ and the \emph{generalized quaternion group} of order~$4n$, respectively. That is,
	$$ D_{2n} = \langle f,x \mid f^2 = x^n = e, \  fxf = x^{-1} \rangle $$
and
	$$ Q_{4n} = \langle f,x \mid x^{2n} = e, \ f^2 = x^n, \ f^{-1} x f = x^{-1} \rangle. $$
\end{notation}

\begin{defn} \ \label{DihTypeDefn}
	\begin{itemize}
	\item  A group~$G$ is of \emph{dihedral type} if it has
		\begin{itemize}
		\item an abelian subgroup~$A$ of index~$2$,
		and
		\item an element~$f$ of order~$2$ (with $f \notin A$),
		\end{itemize}
	such that $f$ inverts every element of~$A$ (i.e., $f^{-1} a f = a^{-1}$ for all $a \in A$).
	\item  A group~$G$ is of \emph{quaternion type} if it has 
		\begin{itemize}
		\item an abelian subgroup~$A$ of index~$2$,
		and
		\item an element~$f$ of order~$4$,
		\end{itemize}
	such that $f$ inverts every element of~$A$. 
	\end{itemize}
 Thus, dihedral groups are the groups of dihedral type in which $A$ is cyclic, while generalized quaternion groups are the groups of quaternion type in which $A$ is cyclic.
 \end{defn}

It is not very difficult to show that Cayley graphs on dihedral groups of small order are hamiltonian:

\begin{lem}[Witte {\cite[Prop.~5.5]{Witte-CayDiags}}] \label{dihedral}
If $n$ has at most three distinct prime factors, then every connected Cayley graph on $D_{2n}$ has a hamiltonian cycle.
\end{lem}

A similar argument also yields a result for other groups of dihedral type:

\begin{prop}[Jungreis-Friedman {\cite[Thm.~5.4]{JungreisFriedman}}] \label{DihType}
 If $G = \integer_2 \ltimes A$ is of dihedral type, and $|A|$ is the product of at
most three primes {\upshape(}not necessarily distinct\/{\upshape)}, then every connected Cayley graph on~$G$ has a hamiltonian cycle. 
 \end{prop}
 
 \begin{proof}
 Let $S$ be a minimal generating set of~$G$.
 Since every element of $fA$ inverts~$A$, it is easy to see that we may assume $S \cap A = \emptyset$ (cf.\ \cite[Thm.~5.3]{Witte-CayDiags}), and that $f \in S$.
	\begin{itemize}
	\item If $A$ is a $p$-group, then \cref{pkSubgrp} applies.
	\item If $A$ is cyclic, then $G$ is dihedral, so \cref{dihedral} applies.
	\end{itemize}
Thus, we may assume $A = \integer_p \times \integer_p \times \integer_q$, where $p$ and~$q$ are distinct primes.

Note that $fS - \{e\}$ must be a minimal generating set of~$A$.
 
 \setcounter{case}{0}
 \begin{case}
 Assume $fS$ contains an element~$x$ of order~$p$.
\end{case}
Then $fS - \{x\}$ must generate a subgroup of order~$pq$ (necessarily cyclic), so $\langle S - \{fx\} \rangle \iso D_{2pq}$; let $(s_1,s_2,\ldots,s_{2pq})$ be a Hamiltonian cycle 
in $\Cay \bigl( D_{2pq}; S-\{fx\} \bigr)$. We may assume $s_{2pq}=f$.  The vertices of the path $(s_1,s_2,\ldots,s_{2pq})\#$ are all in different right cosets of $\langle x \rangle$, so \cref{FGL(notnormal)} implies that $\bigl((s_1,s_2,\ldots,s_{2pq})\#,fx \bigr)^p$ is a hamiltonian
cycle in $\Cay(G;S)$.

\begin{case}
Assume $fS$ does not contain any element of order~$p$.
\end{case}
Then $S = \{f, f_1,f_2\}$, where $f f_1$ and $f f_2$ both have order~$pq$ (and $\langle f f_1, f f_2 \rangle = A =  \integer_p \times \integer_p \times \integer_q$). 
We may assume $p \ge 3$, for otherwise $[G,G] = \integer_q$, so \cref{KeatingWitte} applies. Then, since at least one of any four consecutive integers is relatively prime to $pq$, there exists $k \in \{0,1,2,\ldots,p\}$, such that $(f_2f)^p(ff_1)^k$ generates $\langle f f_1 \rangle$. This means that $(f_2 f)^{p-k} (f_2f_1)^k$ generates $\langle f f_1 \rangle$, so \cref{FGL} implies that $\bigl( (f_2 f)^{p-k}, (f_2f_1)^k \bigr)$ is a hamiltonian cycle in $\Cay \bigl( A; \{f_2f, f_2f_1 \} \bigr)$. Then it is clear that $\bigl( (f_2 , f)^{p-k}, (f_2, f_1)^k \bigr)$ is a hamiltonian cycle in $\Cay \bigl( G; \{f, f_1, f_2\} \bigr)$.
 \end{proof}
 
 \begin{cor} \label{2p2}
 If\/ $|G| = 2p^2$, where $p$ is prime, then every connected Cayley graph on~$G$ has a hamiltonian cycle.
 \end{cor}
 
 \begin{proof}
 Either $[G,G]$ is cyclic of order~$p$ (so
\cref{KeatingWitte} applies) or $G$ is of dihedral type, so \cref{DihType} applies. 
\end{proof}

\begin{cor}[Jungreis-Friedman cf.\ {\cite[Thm.~5.1]{JungreisFriedman}}]  \label{QnHamCyc}
 If $n$ is the product of at most three primes {\upshape(}not necessarily distinct\/{\upshape)}, then every connected Cayley graph on any group of quaternion type of order~$4n$ has a hamiltonian cycle.
 \end{cor}

\begin{proof}
Let $\Cay(G;S)$ be such a Cayley graph, and assume, without loss of generality, that $S$ is a minimal generating set for~$G$. Let $A$ be an abelian subgroup of index~$2$ in $G$, and let $f \in S$, with $f \notin A$. Then $\langle f^2 \rangle$ is a normal subgroup of order~$2$ in~$G$. Furthermore, $G/\langle f^2 \rangle$ is of dihedral type, so \cref{DihType} implies there is a hamiltonian cycle in $\Cay \bigl( G/\langle f^2 \rangle ; S \bigr)$. Therefore, \cref{DoubleEdge} applies with $s = f = t^{-1}$.
\end{proof}

\begin{rem}
If $G$ is a group of dihedral type, and $|G|$ is divisible by~$4$, then B.\,Alspach, C.\,C.\,Chen, and M.\,Dean \cite{AlspachChenDean} have shown that every connected Cayley graph on~$G$ has a hamiltonian cycle. In fact, the Cayley graphs are hamiltonian connected (or hamiltonian laceable when they are bipartite).
\end{rem}

\subsection{Generator in a cyclic, normal subgroup}

The following observation is well known.

\begin{lem}
\label{NormalEasy}
 Let $S$ generate $G$ and let $s \in S$, such that $\langle s \rangle
\normal G$. If
 \begin{itemize}
 \item $\Cay \bigl( G/\langle s \rangle ; S \bigr)$ has a
hamiltonian cycle,
 and
 \item either
 \begin{enumerate}
 \item \label{NormalEasy-Z} 
 $s \in Z(G)$,
 or
 \item \label{NormalEasy-notZ} 
 $Z(G) \cap \langle s \rangle = \{e\}$,
 or
 \item \label{NormalEasy-p}
 $|s|$ is prime,
 \end{enumerate}
 \end{itemize}
 then $\Cay(G;S)$ has a hamiltonian cycle.
 \end{lem}
 
 \begin{proof}
 Let $(s_1,s_2,\ldots,s_n)$ be a hamiltonian cycle in $\Cay \bigl( G/\langle s \rangle ; S \bigr)$, and let $k = |s_1 s_2\cdots s_n|$, so $(s_1,s_2,\ldots,s_n)^k$ is a cycle in $\Cay(G;S)$. 
 
  \pref{NormalEasy-Z} Since $s \in Z(G)$, it is easy to see that $\Cay(G;S)$ contains a spanning subgraph isomorphic to the Cartesian product $P_{n} \times C_{|s|}$ of a path with~$n$ vertices and a cycle with $|s|$ vertices. Since it is easy to see that this Cartesian product is hamiltonian \cite[Cor.\ on p.~29]{ChenQuimpo-hamconn}, we conclude that $\Cay(G;S)$ has a hamiltonian cycle.

 \pref{NormalEasy-notZ} Let $m = |G|/(nk)$. We claim that
 	$$ \bigl( s^{m-1}, s_1, s^{m-1}, s_2,s^{m-1}, \ldots, s^{m-1}, s_n \bigr)^k $$
is a hamiltonian cycle in $\Cay(G;S)$. 

Let 
	$$ \text{$g_i = (s_1s_2\cdots s_i)^{-1}$ for $0 \le i \le n$, so $g_i g_{i+1}^{-1} = s_{i+1}$,} $$
and note that, since $(s_1,s_2,\ldots,s_n)$ is a hamiltonian cycle, we know that 
	$$ \text{$\{1,g_1,g_2,\ldots, g_{n-1}\}$ is a complete set of coset representatives for $\langle s \rangle$ in~$G$.} $$
Then, for any $h \in G$, 
	$$ \text{$\{h,g_1h,g_2h,\ldots, g_{n-1}h\}$ is also a set of coset representatives.} $$
Also, since $\langle s \rangle$ is abelian, we know that if $x$ and~$y$ are elements in the same coset of $\langle s \rangle$, then $s^x = s^y$. Thus, for any $t \in \langle s \rangle$, we have
	$$ \{ t , t^{g_1} ,  t^{g_2} , \ldots , t^{g_{n-1}}  \} = \{ t^h , t^{g_1h} ,  t^{g_2h} , \ldots , t^{g_{n-1}h}  \}, $$
so 
	$$ t t^{g_1}  t^{g_2} \cdots t^{g_{n-1}} =  t^h  t^{g_1h}   t^{g_2h}  \cdots  t^{g_{n-1}h},$$
because both products have exactly the same factors (but possibly in a different order). Since the right-hand product is $(t t^{g_1}  t^{g_2} \cdots t^{g_{n-1}})^h$, and $h$~is an arbitrary element of~$G$, we conclude that $t t^{g_1}  t^{g_2} \cdots t^{g_{n-1}} \in Z(G)$. Since $Z(G)$ has trivial intersection with $\langle s \rangle$, this implies that 
	$$t t^{g_1}  t^{g_2} \cdots t^{g_{n-1}} = e .$$
Therefore
	$$ (s^{m-1}) s_1 (s^{m-1}) s_2 \cdots  (s^{m-1}) s_n 
	=  \bigl( (s^{m-1}) (s^{m-1})^{g_1}  (s^{m-1})^{g_2} \cdots (s^{m-1})^{g_{n-1}}  \bigr) g_n^{-1} 
	= g_n^{-1}. $$
Therefore
	$$ \bigl( (s^{m-1}) s_1 (s^{m-1}) s_2 \cdots  (s^{m-1}) s_n \bigr)^k 
	= g_n^{-k} 
	= (s_1 s_2 \cdots s_n)^k
	= e ,$$
so the walk is closed. Furthermore, since $m = |\langle s \rangle/\langle g_n \rangle|$, it is clear that the walk visits every element of $\langle s \rangle$, and it is similarly easy to see that it visits every element of all of the other cosets. So it visits every element of~$G$.

Since it is also a closed walk of the correct length, we conclude that it is a hamiltonian cycle. 
 
  \pref{NormalEasy-p} 
 Since $|s|$ is prime, either \pref{NormalEasy-Z} or \pref{NormalEasy-notZ} must apply.
 \end{proof}

The following related result is much less obvious.

\begin{thm}[Alspach {\cite[Thm.~3.7]{Alspach-lifting}}]
\label{AlspachLifting}
 Suppose
 \begin{itemize}
 \item $S$ is a generating set of~$G$,
 \item $s \in S$,
 \item $\langle s \rangle \normal G$,
 \item $|G : \langle s \rangle|$ is odd,
 and
 \item $\Cay \bigl( G/ \langle s \rangle; S \bigr)$ has a hamiltonian
cycle.
 \end{itemize}
 Then $\Cay(G;S)$ has a hamiltonian cycle.
 \end{thm}

 A well-known theorem of  B.\,Alspach \cite{AlspachGP} 
 describes exactly which generalized Petersen graphs have a hamiltonian cycle. We need only the following consequence of this very precise result.
 
 \begin{thm}[B.\,Alspach {\cite{AlspachGP}}] \label{GenPet}
 Suppose $X$ is a generalized Petersen graph that is connected, and has $2n$~vertices.
 If $n \not\equiv 0 \pmod{4}$ and $n \not \equiv 5 \pmod{6}$, then $X$ has a hamiltonian cycle.
 \end{thm}

 \subsection{A few small groups}
 
For future reference, we record the existence of hamiltonian cycles in every connected Cayley graph on the groups $S_4$, $A_4 \times \integer_2$, $A_4 \times \integer_3$, and $A_5$. Only a few non-isomorphic Cayley graphs arise on each group, and a computer search could quickly find a hamiltonian cycle in each of them, so, for brevity, we omit some details of the proofs.

\begin{lem}[{}{\cite[Thm.~8.2]{JungreisFriedman}}] \label{S4}
Every connected Cayley graph on the symmetric group $S_4$ has a hamiltonian cycle.
\end{lem}

\begin{proof}
Suppose $S$ is a minimal generating set of $S_4$. Note that $S$ must contain an odd permutation; that is, $S$ contains either a 2-cycle or 4-cycle.  

\setcounter{case}{0}

\begin{case}
Assume $\#S = 2$.
\end{case}
Write $S = \{a,b\}$.

If $a$~is a $4$-cycle, then we may assume 
	$$ \text{$a = (1,2,3,4)$ and $b \in \{(1,2), (1,2,3),(1,2,4,3)\}$.} $$
  In each case, \cref{Stud61} provides a hamiltonian cycle in $\Cay(S_4;S)$.

Now suppose $S$ contains no $4$-cycles. Then we may assume $a = (1,2)$ and $b = (2,3,4)$.  In this case, a hamiltonian cycle is given by 
	$ \bigl( (a, b^2)^2, (a, b^{-2})^2 \bigr)^2 $.

\begin{case}
Assume $\#S \ge 3$.
\end{case}
Since $S$ is minimal, it is easy to see that $\#S = 3$; write $S = \{a,b,c\}$.

\begin{subcase}
Assume $a = (1,2)(3,4)$.  
\end{subcase}
Let $N$ be the normal subgroup of order~$4$ that contains~$a$. 
Then, since $\langle S \rangle = S_4$, we have $\langle b,c \rangle N = S_4$. Furthermore, since the action of $S_4/N$ on~$N$ is irreducible, the minimality of~$S$ implies $ \langle b,c \rangle \cap N$ is trivial. So $\langle b , c \rangle \iso S_4/N \iso S_3$. Then, conjugating by a power of $(1,3,2,4)$ (which centralizes~$a$), there is no harm in assuming that $\langle b,c \rangle = S_3$. So we may assume $\{b,c\}$ is either
	$$ \text{$\{(1,2), (2,3)\}$ or $\{(1,3),(2,3)\}$ or $\{(1,2,3), (1,2)\}$ or $\{(1,2,3),(2,3)\}$} .$$
 If $(2,3) \in S$, then \cref{Stud71} applies with $s_1 = (1,2)(3,4)$ and $s_2 = (2,3)$.  In the
remaining case, let $b = (1,2,3)$ and $c = (1,2)$, and let $L$ be any hamiltonian path in $\Cay \bigl( A_4; \{a,b\} \bigr)$ from~$e$ to~$b$. Then $(L,c)^2$ is a hamiltonian cycle in $\Cay(S_4;S)$.

\begin{subcase}
Assume $S$ does not contain any even permutation of order~$2$.
\end{subcase}
Then, since $S$ is a 3-element, minimal generating set, it is not difficult to see that $S$ cannot
contain a $4$-cycle.  So $S$ consists entirely of $2$-cycles and $3$-cycles. However, it is known that there is a hamiltonian cycle in $\Cay(S_n:S)$ whenever $S$ consists entirely of $2$-cycles (see the discussion and references on p.~622 of \cite{Savage-GraySurvey}), so we may assume that $S$ contains at least one $3$-cycle. Then, up to automorphism, we have
	$$ S = \{(1,2,3),(1,2,4),(1,2)\} .$$
Let $a = (1,2,3)$, $b = (1,2,4)$, and $c = (1,2)$, and let $L$ be any hamiltonian path in $\Cay \bigl( A_4; \{a,b\} \bigr)$ from~$e$ to~$b$. Then $(L,c)^2$ is a hamiltonian cycle in $\Cay(S_4;S)$.
\end{proof}

We actually need only the cases $p = 2$ and $p = 3$ of the following result, but the general case is no more difficult to prove.

\begin{lem}[{}{Jungreis-Friedman \cite[Thm.~7.4]{JungreisFriedman}}] \label{A4xZp}
If $p$~is prime, then every connected Cayley graph on $A_4 \times \integer_p$ has a hamiltonian cycle.
\end{lem}

\begin{proof}
Suppose $S$ is a minimal generating set for the group $A_4 \times \integer_p $, and 
let $z$ be a generator of~$\integer_p$. 
Note that every minimal generating set of $A_4$ is  of the form
$\{(1,2,3),(1,2,4)\}$ or $\{(1,2,3),(1,2)(3,4)\}$ (up to automorphism), so $S$ has either $2$ or~$3$
elements. 

\setcounter{case}{0}

\begin{case}
Assume $\#S = 2$.
\end{case}
Write $S = \{a,b\}$.

\begin{subcase}
Assume $a = (1,2)(3,4)z$.
\end{subcase}
We may assume $p = 2$, for otherwise \cref{DoubleEdge} applies with $N = \integer_p$. 
Let $b$ be the second element of~$S$; we may assume $b$~is either $(1,2,3)$ or $(1,2,3)z$. 
	\begin{itemize}
	\item If $b = (1,2,3)$, then $ \bigl((b^2,a)^2,(b^{-2},a)^2\bigr)^2 $ is a hamiltonian cycle.
	\item  If $b = (1,2,3)z$, then $(b^5,a,b^{-5},a)^2$ is a hamiltonian cycle.
	\end{itemize}
\begin{subcase}
Assume $(1,2)(3,4)z \notin S$.
\end{subcase}
We may assume $(1,2,3)z \in S$.
Then we may assume $p = 3$, for otherwise it is not difficult to verify that \cref{Stud61} applies with $s_1 = (1,2,3)z$.

Let $a = (1,2,3)z$, and let $b$ be the other element of~$S$. Since $\{(1,2,3)z, (1,2)(3,4) \rangle \neq A_4 \times \integer_3$, we must have $b = (1,2,4) z^i$ for some~$i$. By applying an automorphism, we may assume $i = 0$. Then a hamiltonian cycle is given by
	$$ \bigl( a^{-2},b^{-2},a^2,b,(a^{-2},b^2)^2,a^2,b,a^{-2},b^{-2},a,b,a^2,b^{-1},a^{-2},b^2,a^{-2},b,a^{-1},b \bigr) .$$

\begin{case}
Assume $\#S=3$.
\end{case}
We may assume $S \cap \integer_p = \emptyset$, for otherwise \fullcref{NormalEasy}{Z}  applies.
Then, from the minimality of~$S$, it is not difficult to see that we must have $p = 3$ and (after applying an automorphism) $S$ contains both $(1,2,3)$ and $(1,2,3)z$, so \cref{DoubleEdge} applies (with $N = \integer_p$).
\end{proof}

Our proof \cite{M2Slovenian-A5} of the following result consists of two pages of unilluminating case-by-case analysis, so we omit it.

\begin{lem}[{}{\cite{M2Slovenian-A5}}] \label{A5}
Every connected Cayley graph on the alternating group $A_5$ has a hamiltonian cycle.
\end{lem}

 \subsection{Some facts from group theory}
  

The following well-known consequence of Sylow's Theorems will be used several times.

\begin{lem}[{\cite[Thm.~25.1]{Gallian-AlgText}}] \label{NoMod1->Normal}
Suppose $G$ is a finite group, $p$~is a prime number, and $P$ is a Sylow $p$-subgroup of~$G$. If\/ $1$~is the only divisor~$k$ of $|G|/|P|$, such that $k \equiv 1 \pmod{p}$, then $P$~is a normal subgroup of~$G$.
\end{lem}

We recall a few basic facts about the Frattini subgroup.

\begin{defn}
The \emph{Frattini subgroup} of a finite group~$G$ is the intersection of all the maximal subgroups of~$G$. It is denoted $\Phi(G)$.
\end{defn}

\begin{prop}[{cf.~\cite[Thms.~5.1.1 and 5.1.3]{Gorenstein-FinGrps}}]
Let $S$ be a minimal generating set of a finite group~$G$. Then:
	\begin{enumerate}
	\item $\Phi(G)$ is a normal subgroup of~$G$.
	\item $S \cap \Phi(G) = \emptyset$.
	\item $S$ is a minimal generating set of $G/\Phi(G)$.
	\item If $G$ is a $p$-group, then\/ $\Phi(G) = [G,G] \cdot \langle g^p \mid g \in G \rangle$.
	\end{enumerate}
\end{prop}

And we also recall some very basic facts about Hall subgroups.

\begin{defn}
A subgroup~$H$ of a finite group~$G$ is a \emph{Hall subgroup} if $|H|$ is relatively prime to $|G|/|H|$.
\end{defn}

\begin{prop}[{\cite[Thm.~6.4.1(i)]{Gorenstein-FinGrps}}] \label{Hall-solv}
Let 
	\begin{itemize}
	\item $G$ be a finite group that is solvable,
	and
	\item $k$ be a divisor of\/~$|G|$, such that $k$ is relatively prime to\/ $|G|/k$.
	\end{itemize}
Then $G$ has a Hall subgroup whose order is precisely~$k$.
\end{prop}

\section{Groups of order $8p$}\label{8p-section}

We begin with a special case:

\begin{lem} \label{8pNotNormal}
   If\/ $|G| = 8p$, where $p$ is prime, and the Sylow $p$-subgroups are not normal, then every connected Cayley graph on~$G$ has a hamiltonian cycle.
\end{lem}

\begin{proof}
Let $P$ be a Sylow $p$-subgroup of~$G$. 
Sylow's Theorem \pref{NoMod1->Normal} implies $p$ is either $3$ or~$7$.

\setcounter{case}{0}

\begin{case}
Assume $p = 3$.
\end{case}
The normalizer $N_G(P)$ is not all of~$G$, so $| G : N_G(P)
| = 4$.  Letting $G$ act on the cosets of $N_G(P)$ by translation yields a homomorphism from~$G$ to~$S_4$.  Then either $G = S_4$ (so \cref{S4} applies), or $N_G(P)$ contains a normal subgroup~$N$ of~$G$,
which must be of order~$2$, and thus $N \subseteq Z(G)$.

Since $|G/N| = 12$, and the Sylow $3$-subgroup is not normal, we have $G/N \iso A_4$. We may assume $G \not\iso A_4 \times \integer_2$ (otherwise \cref{A4xZp} applies),  so it is easy to see that $G \iso \integer_3 \ltimes Q_8$. Since $G/N \iso A_4 \iso \integer_{3} \ltimes (\integer_2 \times \integer_2)$, the proof of \cref{Z3kx(Z2xZ2)} tells us that 
	$ ( s_1^{2}, s_2^{-1}, s_1^{-2}, s_2)^{2}$
is a hamiltonian cycle in $\Cay(G/N;S)$. Its endpoint in~$G$ is $[s_1,s_2]^2$. This generates~$N$ (because the square of any element of $Q_8 - \{\pm1\}$ is nontrivial), so the Factor Group Lemma \pref{FGL} provides a hamiltonian cycle in $\Cay(G;S)$.

\begin{case}
Assume $p = 7$.
\end{case}
It is not difficult to see that we must have $G= \integer_7 \ltimes (\integer_2)^3$. Let $x$ and~$y$ be nontrivial elements of $\integer_7$ and $(\integer_2)^3$, respectively. Then, up to automorphism (and replacing generators by their inverses), it is clear that every minimal generating is of the form $\{x, x^iy\}$ with $i \in \{0, 1, 2 , 4\}$. Furthermore, since $x \equiv (x^4y)^2 \pmod{(\integer_2)^3}$, an automorphism carries the generating set $\{x, x^4 y\}$ to $\{x, x^2 y\}$; so we may assume $i \neq 4$. Here are hamiltonian cycles for the three remaining cases.
	\begin{itemize}
	\item[] $i = 0$:  $\bigl( (x^6,y)^2,(x^{-6},y)^2 \bigr)^2$,
	\item[] $i = 1$: $(x^6,xy, xy)^7$,
	\item[] $i = 2$: $\Bigl(( x^6,x^2 y)^2,\bigl( x^{-6}, (x^2 y)^{-1})^2 \Bigr)^2$.
	\end{itemize}
\end{proof}

\begin{prop}[Jungreis-Friedman {\cite[Thm.~8.5]{JungreisFriedman}}] \label{8p}
   If\/ $|G| = 8p$, where $p$ is prime, then every connected Cayley graph on~$G$ has
   a hamiltonian cycle.
\end{prop}

\begin{proof}
Let $S$ be a minimal generating set of~$G$. We may assume $p > 2$, for otherwise $|G| = 16$ is a prime power, so \cref{{pk}} applies. We may also assume $G$ has a normal Sylow $p$-subgroup, for otherwise \cref{8pNotNormal} applies. Hence, $G = P_2 \ltimes \integer_p$ where $P_2$ is a Sylow $2$-subgroup of $G$. 
  
 We assume the commutator subgroup of~$G$ is not cyclic of prime order (otherwise \cref{KeatingWitte} applies). Hence, $P_2$ is nonabelian, and acts nontrivially on $\integer_p$. The only nonabelian groups of order~$8$ are $D_8$ and~$Q_8$.

\setcounter{case}{0}

   \begin{case} 
   Assume ${P_2 \cong Q_8}$.
   \end{case}
   Since $\Aut \integer_p$ is cyclic, and the only nontrivial cyclic quotient of~$Q_8$ is~$\integer_2$, it must be the case that the kernel of the action of $Q_8$ on $\integer_p$ is a subgroup of order~$4$. The subgroups of order~$4$ in~$Q_8$ are cyclic, so it is easy to see that $G$ is of quaternion type. Hence, \cref{QnHamCyc} applies.

   \begin{case} 
   Assume ${P_2 \cong D_{8}}$.
   \end{case}
   The argument of the preceding case shows that the kernel of the action of $D_8$ on $\integer_p$ is of order~$4$. If it is cyclic,
   then $G \cong D_{8p}$, which is covered by \cref{dihedral}. 
   
   Henceforth, we assume the kernel is not cyclic, so
   $$G \iso \{f, x, z\, |\, f^2 = x^4 = z^p = 1, x^f = x^{-1}, z^f = z, z^x = z^{-1}\} .$$
Note that, by \cref{NormalEasy}, we may assume $S$ does not contain any element of $\langle z \rangle$.

Note that:
	\begin{itemize}
	\item $x^2$ is an element of order~$2$ in $Z(G)$, so $x^2$ belongs to every Sylow $2$-subgroup, 
	and
	\item $x^2$ is in the Frattini subgroup of $D_8 \iso G/\langle z \rangle$, so $x^2$ also belongs to every maximal subgroup that contains $\langle z \rangle$. 
	\end{itemize}
Since every maximal subgroup of~$G$ either  is a Sylow $2$-subgroup or contains $\langle z \rangle$, we conclude that $x^2 \in \Phi(G)$.  Since $S$ is minimal, this tells us that $S$ is a minimal generating set of $G/\langle x^2 \rangle$.
Therefore, we may assume $S$ does not contain any element~$s$ such that $s^2 = x^2$. 
(Otherwise, \cref{DoubleEdge} applies with $t = s^{-1}$.) This means that $S$ does not contain any element of the form $x^{\pm 1} z^n$.

Then, since $S$ must generate $G/\langle x^2, z \rangle$, we conclude that $S$ (or $S^{-1}$) contains elements of the form $fx z^\ell$ and $f x^{2m} z^n$. We may assume:
	\begin{itemize}
	\item $\ell = 0$ (conjugating by a power of~$z$),
	\item $m = 0$ (because we can conjugate by a power of~$x$ and replace $x$~with $x^{-1}$ if necessary),
	and
	\item $n = 0$ (otherwise, we may apply \cref{DoubleEdge} with $N = \langle z \rangle$ and $t^{-1} = s = f z^n$).
	\end{itemize}
Thus, $S$ contains both $f$ and $fx$.

Now, in order to generate $z$, the set $S$ must contain either $x^2z$ or $fx^iz$, for some~$i$ (up to replacing $z$ with one of its powers).

   \begin{subcase} 
   Assume ${x^2z}$ in ${S}$.
   \end{subcase}
   We must have $S = \{f, fx, x^2z\}$ and \cref{FGL(notnormal)} shows that $(fx,f,fx,x^2z)^{2p}$ is a hamiltonian cycle, since $fx\cdot f\cdot fx\cdot x^2z = fz$ is of order $2p$, and the vertices
   	$$ \text{$e$, $fx$, $fx \cdot f = x^3$, and $x^3  \cdot fx = f x^2$}$$
are all in different right cosets of $\langle f, z \rangle$.

   \begin{subcase} 
   Assume there is no element of the form ${x^2z^i}$ in ${S}$.
   \end{subcase}
   Then $fx^iz \in S$. Note that $i$ must be odd, for otherwise $\langle z \rangle \subset \langle fx^i z \rangle$, so $\langle fx, fx^i z \rangle = G$, contradicting the minimality of~$S$. 
 Thus, we have 
 	$$ \text{$S = \{f, fx, fxz\}$ or $S = \{f, fx, fx^3z\}$.} $$
In the former case,
   \cref{DoubleEdge} applies (letting $s = fx$, $t = fxz$, and $N = \langle z \rangle$). 
   
We may now assume $S = \{f, fx, fx^3z\}$. Because we have $G/\integer_p \iso D_8$, it is easy to check that
   	$$ (fx, fx^3 z, fx, f, fx^3z, fx, fx^3z, f )$$
is a hamiltonian cycle in $\Cay(G/\integer_p; S)$, and we have
	$$ (fx)( fx^3 z)( fx)( f)( fx^3z)( fx)( fx^3z)( f ) = z^3 .$$
If $p \ge 5$, this product generates $\integer_p$, so the Factor Group Lemma \pref{FGL} tells us there is a hamiltonian cycle in $\Cay(G;S)$.

If $p = 3$, a hamiltonian cycle is given by
	\begin{align*}
	 \bigl( (f,fx)^2,fx^3 z,f,fx,f,fx^3 z,(fx,f)^2,fx^3 z,f,fx^3 z,fx,f,fx^3 z,f,fx,fx^3 z,f,fx^3 z \bigr) 
	 . \end{align*}
\end{proof}


\section{Groups of order $3p^2$} \label{3p2Sect}

\begin{prop} \label{3p2}
 If\/ $|G| = 3p^2$, where $p$ is prime, then every connected Cayley graph on~$G$ has
a hamiltonian cycle.
 \end{prop}

\begin{proof}
Let $S$ be a minimal generating set of~$G$, and let $P$ be a Sylow $p$-subgroup of~$G$. We may assume $p \ge 5$, for
otherwise either $|G| = 12$ is of the form~$4p$, so \cref{4p} applies, or $|G| = 3^3$ is a prime power, so \cref{pk} applies. Hence, Sylow's Theorem \pref{NoMod1->Normal} tells us $P \normal G$. Note that $G/P \iso \integer_{3}$ is
abelian,
so $[G,G] \subset P$. So we may assume $P \iso \integer_p \times
\integer_p$,
for otherwise \cref{KeatingWitte} applies. Thus, we may assume $G =
\integer_{3} \ltimes (\integer_p \times \integer_p)$, and $[G,G] = \integer_p \times \integer_p$.

\setcounter{case}{0}

\begin{case}
 Assume ${|S| = 3}$.
 \end{case}
 Write $S = \{s,t,u\}$. We may assume $s \notin \integer_p \times \integer_p$, so $|s| = 3$. Because $S$ is minimal,
we see that $\langle s,t \rangle$ and $\langle s,u \rangle$ each
have order~$3p$. Let $N$ be the unique subgroup of order~$p$ in $\langle
s,t \rangle$, and note that $N$ is normal in~$G$ (because it is normalized both by~$s$ and by the abelian group $\integer_p \times \integer_p$).

	\begin{itemize}
	\item If $t \in N$, then \cref{AlspachLifting} applies.
	\item If $t \notin N$, then we may assume $s \equiv t \pmod{N}$ (by replacing
$t$ with~$t^{-1}$, if necessary). So \cref{DoubleEdge} applies.
	\end{itemize}

\begin{case}
 Assume ${|S| = 2}$.
 \end{case}
 Write $S = \{s,t\}$. 
 
 If $s$ and~$t$ both have order~$3$, then we may
assume $s \equiv t \pmod{\integer_p \times \integer_p}$ (by replacing
$t$
with~$t^{-1}$, if necessary). Then $s t^{-1}$ is contained in the normal
$p$-subgroup $\integer_p \times \integer_p$, so \cref{pkSubgrp}
applies.

 We may now assume $|s| = 3$ and $|t| = p$.

Let us determine the action of~$s$ on $\integer_p \times
\integer_p$.
 \begin{itemize}
 \item Define a linear transformation~$T$ on $\integer_p \times
\integer_p$ by $T(v) = s^{-1} v s$,
 \item let $m(x)$ be the minimal polynomial of~$T$,
 and
 \item let $u = T(t) = s^{-1} t s$.
 \end{itemize}
 Note that:
 \begin{itemize}
 \item Because $|s| = 3$, we know $T^3 = I$, so $m(x)$ divides $x^3 - 1
=
(x-1)(x^2 + x + 1)$.
 \item Since $|[G,G]| = p^2$, we know that $1$~is not an eigenvalue of~$T$.
 \item Because $\langle s,t \rangle = G$, we know $u = T(t) \notin
\langle
t \rangle$, so the minimal polynomial of~$T$ has degree~$2$ (and
$\{t,u\}$ is a basis of $\integer_p \times \integer_p$).
 \end{itemize}
 We conclude that the minimal polynomial of~$T$ is $x^2 + x + 1$. Thus,
with respect to the basis $\{t, u\}$ of $\integer_p \times
\integer_p$, we have $T(i,j) = (-j, i-j)$. In other words,
 $ s^{-1} (t^i u^j) s
 = t^{-j} u^{i-j} $,
 so
 \begin{equation} \label{3p2-multiedges}
 t^i u^j s
 = s \bigl( s^{-1} (t^i u^j) s \bigr)
 = s (t^{-j} u^{i-j})
 \in \langle s \rangle t^{-j} u^{i-j} .
 \end{equation}

The quotient multigraph $\langle s \rangle \backslash \Cay(G;S)$ has the following
properties:
 \begin{itemize}
 \item it has $p^2$ vertices;
 \item it is regular of valency~$4$;
 \item the vertices are the ordered pairs $(i,j)$, where 
 	$$-\frac{p-1}{2} \le i,j\le \frac{p-1}{2} ;$$
	and
 \item the vertex $(i,j)$ is adjacent to: 
 	$$ \text{$(i+1,j)$; $(i-1,j)$;
$(-j,i-j)$; and $(j-i,-i)$,} $$
where calculations are performed modulo $p$
\see{3p2-multiedges}.
 \end{itemize}
Furthermore, we observe that for any prime $p$, there are precisely two multiple edges:
	\begin{itemize}
	\item If $p = 3k+1$, then, with
$j=k$ and $i=2j$, we have 
	$$(-j,i-j)=(-k,k)=(2k+1,k)=(i+1,j) ,$$
and, with $j=-k$ and $i=2j$, we have 
	$$(-j,i-j)=(k,-k)=(-2k-1,-k)=(i-1,j) .$$
	\item If $p=3k+2$, then, with $j=k+1$
and $i=2j$, we have 
	$$(-j,i-j)=(-k-1,k+1)=(2k+1,k+1)=(i-1,j) ,$$
and, with
$j=-(k+1)$ and $i=2j$, we have
	$$(-j,i-j)=(k+1,-(k+1))=(-2k-1,-(k+1))=(i+1,j) .$$
	\end{itemize}

\begin{figure}
 \begin{center}
 \includegraphics{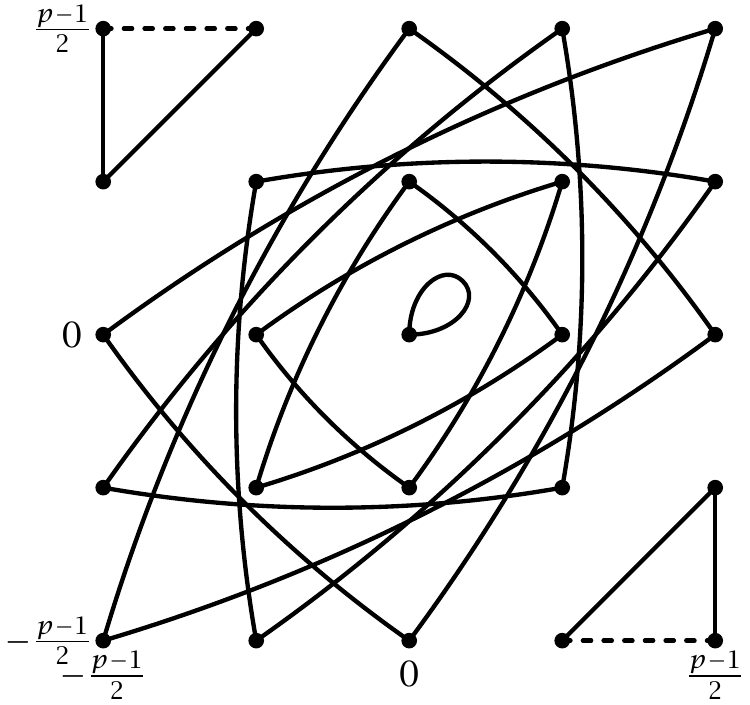}
 \end{center}
 \caption{The $s$-edges in the multigraph when $p = 5$. The $t$-edges
(not drawn) are horizontal. Thus, there are two double edges (dashed, at
top left and bottom right) in the multigraph.}
 \end{figure}

We now construct a hamiltonian cycle in the multigraph.
Beginning at the vertex $(1,1)$, we proceed along the following sequence of edges (where $\{n\}$ denotes the remainder when~$n$ is divided by~$p$, so $0 \le \{n\} < p$):
\begin{align*}
&\Bigl( \bigl[(t^{-1})^{\{3j-1\}}, s^{-1}, t^{\{-3j-1\}}, s^{-1}\bigr]_{j=1}^{(p-3)/2} ,
\\& \quad \bigl[ (t^{-1})^{(p-5)/2}, s^{-1}, t^{(p+1)/2} , s\bigr], 
\\&\quad \quad \bigl[t^{p-1}, s^{-1}\bigr], 
\\& \quad \quad \quad \bigl[t^{\{3j-1\}}, s, (t^{-1})^{\{-3j-1\}}, s^*\bigr]_{j=-(p-1)/2}^{-k-1}, 
\\&\quad \quad \quad \quad \bigl[(t^{-1})^{\{3j-1\}}, s^{-1}, t^{\{-3j-1\}}, s^{-1}\bigr]_{j=-k}^{-2} , 
\\& \quad \quad \quad \quad \quad \bigl[(t^{-1})^{p-4}, s^{-1}, t^2, s\bigr]
\Bigr)
. \end{align*}
Here $s^*$ indicates $s$ unless $j=-k-1$ and $p=3k+1$, in which case it indicates $s^{-1}$.

We make a few observations that will aid the reader in verifying that this is indeed a hamiltonian cycle.  We will see that each portion enclosed within square brackets traverses all of the vertices in some row of the multigraph, and passes on to the next row; the parameter $j$ represents the row being traversed. For convenience, we will use $\overline{n}$ to denote the integer congruent to~$n$ modulo~$p$ that is between $-(p-1)/2$ and $(p-1)/2$ (inclusive).

In the first portion enclosed within square brackets, we begin each row at a vertex of the form $(\overline{2j-1},j)$.  The $t^{-1}$ edges cover every vertex between this and $(-j,j)$, moving leftwards; $s^{-1}$ takes us to the vertex $(\overline{2j},j)$, and the $t$ edges cover every vertex between this and $(\overline{-j-1},j)$, completing the row.  The $s^{-1}$ edge then takes us to $(\overline{2j+1},j+1)$, which has the required form to continue this pattern.  In this way, we cover all of the vertices from rows 1 through $(p-3)/2$.  

Then, the second portion enclosed within square brackets covers the vertices of row $(p-1)/2$ in the same way, except that it ends with an $s$-edge, taking us from $\bigl( (p-1)/2,(p-1)/2 \bigr)$ to $\bigl( -(p-1)/2,0 \bigr)$.

The third portion enclosed within square brackets takes us through the vertices of row 0, ending at $\bigl( (p-1)/2,0 \bigr)$, and to vertex $\bigl( -(p-1)/2,-(p-1)/2 \bigr)$, which has the form $(\overline{-j+1},j)$.  

In the fourth portion enclosed within square brackets, we use the same pattern on rows $-(p-1)/2$ through $-k-1$ that we used on rows $k+1$ through $(p-1)/2$, rotated by 180 degrees.  So in each row we begin at a vertex of the form $(\overline{-j+1},j)$, move right until we reach $(\overline{2j},j)$, then jump to $(-j,j)$, move left to the vertex $(\overline{2j+1},j)$ (completing the row), and pass to vertex $(-j,j+1)$ in the next row, which again has the required form to continue.  The $s^*$~edge ends this portion at the vertex $(k+1,-k)$ if $p=3k+2$, and at $(k,\overline{2k+1})$ if $p=3k+1$; both of these are equal to $(\overline{-2k-1},-k)$.  This has the form $(\overline{2j-1},j)$ again.

Now the fifth portion uses the same pattern we began with, to cover the vertices in rows $-k$ through $-2$.  

The final portion covers row $-1$, using the same pattern, finishing at the vertex $(0,-1)$.  Now edge $s$ takes us to $(1,1)$, completing the hamiltonian cycle.

Notice that in each row (except row 0), this cycle traversed the unique $s$-edge whose endpoints were both in that row.  In particular, this cycle traversed each of the multi-edges noted above.
So \cref{MultiDouble}
implies
there is a hamiltonian cycle in $\Cay(G;S)$.
 \end{proof}


\section{Groups of order $4p^2$} \label{4p2Sect}

We begin with a special case.

\begin{lem}[Jungreis-Friedman {\cite[Thm.~7.3]{JungreisFriedman}}]\label{D2pxD2p}
 If $p$ is prime, then every connected Cayley graph on $D_{2p}\times
D_{2p}$ has a hamiltonian cycle.
 \end{lem}

\begin{proof}
Let $S$ be a minimal generating set of $G=D_{2p}\times D_{2p}$, where
$p$ is prime.  We may assume $p \ge 3$, for otherwise
$|G| = 16$ is a power of~$2$, so \cref{{pk}} applies.

Notice that the elements of $S$ cannot
consist exclusively of pairs of nontrivial reflections, together with
pairs of (possibly trivial)
rotations, since such a set cannot generate the element $(f, e)$ of
$D_{2p}\times D_{2p}$
(for any reflection $f$).
This is because getting $f$ in the first coordinate requires taking
the product of an odd number of
elements that are pairs of nontrivial reflections, while getting $e$ in
the second coordinate
requires taking the product of an even number of such elements.  
Therefore, as $S$
must
contain an element with a reflection in its first coordinate, we may
assume that either
$(f, e)\in S$ or $(f, x')\in S$, for some reflection $f$ and some
nontrivial rotation $x'$.

\setcounter{case}{0}
\begin{case} \label{D2pxD2pPf-S=2}
Assume ${|S| = 2}$.
\end{case}
In order to generate
the
entire group, $S$~must  include either two reflections, or a reflection
and a rotation, within each of the dihedral factors.  Up to
automorphism, there are only two generating sets that satisfy this
condition: $\{(x, f'),(f,x')\}$, or $\{(f_1,x'),(f_2,f')\}$, where  $x,
x'$
are nontrivial rotations and $f, f', f_1, f_2$ are reflections and $f_1\neq f_2$.

We now verify that each of the above generating sets satisfies the
conditions of \cref{Stud61},
so that by that \lcnamecref{Stud61},
$\Cay(G;S)$ does indeed have a hamiltonian cycle.

\begin{subcase}
Assume ${s_1=(x,f')}$ and ${s_2=(f,x')}$.  
\end{subcase}
Letting $\gamma=[s_1,s_2]$, we have
	$$\gamma= \bigl( x^{-1}fxf, f'(x')^{-1}f'x' \bigr)= \bigl(x^{-2},(x')^2\bigr) .$$
So
\begin{itemize}
\item $2|s_1||\gamma|=2(2p)(p)=4p^2=|G|$.
\item $\langle s_1\rangle \langle \gamma \rangle
= \bigset{ \bigl( x^ix^{-2j},(f')^i(x')^{2j} \bigr) }{ i, j\in \integer}$,
which never has $f$ in the first coordinate, so $s_2$ is not in this
set.
\item If $\gamma^i=s_1^j$ for some $i,j$, then $(x')^{2i}=(f')^j$, so
we must have
$(x')^{2i}=(f')^j=e$, so $i\equiv 0 \pmod{p}$. But then $x^{-2i}=e$, so
$\gamma^i=(e,e)$.  Therefore, $\langle \gamma\rangle \cap \langle
s_1\rangle =\{e\}$.
\item If $\gamma^i=s_2^{-1}s_1^js_2$ for some $i,j$, then
$(x')^{2i} = (x')^{-1}(f')^jx'$,
so $(x')^{2i} = e = (f')^j$, so $i\equiv 0 \pmod{p}$.
But then $x^{-2i}=e$, so
$\gamma^i=(e,e)$.  Therefore,
$\langle \gamma\rangle\cap(s_2^{-1}\langle s_1\rangle s_2)=\{e\}$.
\end{itemize}

\begin{subcase}
Assume
${s_1=(f_1,x')}$ and ${s_2=(f_2,f')}$, with ${f_1\neq f_2}$.
\end{subcase}
We may assume $f_2=f_1x$.
Then with $\gamma=[s_1,s_2]$, we have
$$\gamma= \bigl( f_1(f_1x)f_1(f_1x), (x')^{-1}f'x'f' \bigr)= \bigl(x^{2},(x')^{-2} \bigr) .$$
So
\begin{itemize}
\item $2|s_1||\gamma|=2(2p)(p)=4p^2=|G|$;
\item $\langle s_1\rangle \langle \gamma \rangle
=\bigset{ \bigl( f_1^ix^{2j},(x')^i(x')^{-2j}) }{ i,j \in \integer}$,
which never has $f'$ in the second coordinate, so $s_2$ is not in this
set.
\item If $\gamma^i=s_1^j$ for some $i,j$, then $x^{2i}=f_1^j$, so we
must have
$x^{2i}=e$, so $i\equiv 0 \pmod{p}$. But then $(x')^{-2i}=e$, so
$\gamma^i=(e,e)$.  Therefore, $\langle \gamma\rangle \cap \langle
s_1\rangle =\{e\}$.
\item If $\gamma^i=s_2^{-1}s_1^js_2$ for some $i,j$, then
$x^{2i} = f_2f_1^jf_2$,
so $x^{2i}=e$, so $i\equiv 0 \pmod{p}$. But
then $(x')^{-2i}=e$, so
$\gamma^i=(e,e)$.  Therefore,
$\langle \gamma\rangle\cap(s_2^{-1}\langle s_1\rangle s_2)=\{e\}$.
\end{itemize}

\begin{case}
Assume ${|S| = 3}$.
\end{case}
In what follows, we will be applying \cref{direct-prod-cor}
repeatedly.
We will verify some of its conditions here, so that only one condition
will need to be
checked each time we use it below.  Namely, we already know that  $S$ is minimal and $|S|\ge 3$. Furthermore, we assume, by induction on $|G|$, that  $\Cay \bigl( \langle S-\{s_1\}\rangle
; S-\{s_1\} \bigr)$ has a hamiltonian cycle, for any $s_1 \in S$.  Thus, in order to apply the 
\lcnamecref{direct-prod-cor}, all that remains is to verify that there exist two distinct
generators $s_1,s_2\in S$ such that $|s_1s_2|=|G|/|\langle
S-\{s_1\}\rangle|$ is
prime.

\begin{subcase} \label{SDisjD2pxe}
Assume $S$ is disjoint from $D_{2p} \times \{e\}$ and $\{e\} \times D_{2p}$.
\end{subcase}
The discussion preceding \cref{D2pxD2pPf-S=2} implies that we may
assume $s = (f, x') \in S$, for some reflection $f$ and some nontrivial rotation $x'$. 

The generating set~$S$ must also be an element whose second coordinate is a reflection. By the assumption of this \lcnamecref{SDisjD2pxe},
the first coordinate cannot be trivial. And it also cannot be a nontrivial rotation (because $S$ is minimal and $|S| = 3$). Therefore, it is a reflection. From the minimality of~$S$, we conclude that it is~$f$: that is, $s_1 = (f,f') \in S$, for some reflection~$f'$.

Now, let $s_2 = (y, y')$ be the third element of~$S$. To generate~$G$, we must have $y \neq f$. (And $y'$ is nontrivial, from the assumption of this \lcnamecref{SDisjD2pxe}.) 
Note that if $y$ is a reflection, then either $\langle s_2, s \rangle = G$ or $\langle s_2, s_1 \rangle = G$, depending on whether $y'$ is a rotation or a reflection. Thus, the minimality of~$S$ implies that $y$ is a rotation~$x$ . Then, since $\langle s, (x, f'') \rangle = G$ for any rotation~$f''$, the minimality of~$S$ implies that $y'$ is a rotation. Thus, $s_2 = (x,x'') \in S$, for some nontrivial rotations $x$ and~$x''$. Then 
	$ |s_1s_2| = |(fx, f'x'')| = 2$
and
	$ \langle S - \{s_1\} \rangle = D_{2p} \times \integer_p $
has index~$2$, so \cref{direct-prod-cor} provides a hamiltonian cycle in $\Cay(G;S)$.

\begin{subcase} 
Assume ${(f,e) \in S}$.
\end{subcase}
As there must be a reflection in the second coordinate
of some element of $S$, we have either $(x, f')\in S$, or $(e, f')\in S$, or $(f_1,f')\in S$.

\begin{subsubcase} \label{elt-order-2p}
Assume ${(x,f') \in S}$, with $|x| = p$.
\end{subsubcase}
Then $(x, f')$ generates a subgroup of $G$ of order $2p$.  Let
$s_1=(f,e)$. Since
	$$ [s_1,G] = \langle (x,e) \rangle \subset \langle (x,f') \rangle \subset \langle S - \{s_1\} \rangle ,$$
we know that $s_1$ normalizes $\langle S - \{s_1\} \rangle$, so 
$G = \langle s_1 \rangle \langle S - \{s_1\} \rangle$.
Furthermore, because $S$ is a minimal generating set, we know $s_1 \notin \langle S - \{s_1\} \rangle$. Since $|s_1| = 2$ is prime, this implies $\langle s_1 \rangle \cap \langle S - \{s_1\} \rangle = \{e\}$, so 
$|G| = |\langle s_1 \rangle| \cdot |\langle S - \{s_1\} \rangle|$. 
Therefore $|G| / |\langle S - \{s_1\} \rangle| = 2$.
Also, with $s_2=(x,f')$, we have $|s_1s_2|=2$.
So \cref{direct-prod-cor} tells us that $\Cay(G;S)$
is hamiltonian.

\begin{subsubcase} \label{abelian-order2-generator}
Assume either ${(e,f')\in S}$, or ${(f,f')\in
S}$.
\end{subsubcase}
Let $r_1=(f,e)$ and let $r_2$ be either $(e,f')$ or $(f, f')$, the other
element
that we know to be in $S$.  
Note that, because $|r_1r_2|=2$,
	\begin{align} \label{Sr1Index2}
	\begin{matrix}
	\text{if either $\langle S-\{r_1\}\rangle$ or $\langle S-\{r_2\}\rangle$ has index 2 in $G$,}
	\\
	\text{then \cref{direct-prod-cor} tells us $\Cay(G;S)$ is hamiltonian.} 
	\end{matrix}
	\end{align}

We claim that we may assume no element of $S$ consists of a nontrivial rotation in one
coordinate
together with a reflection in the other.
To see this, observe, first of all, that we are not in \cref{elt-order-2p}, so the reflection
cannot be in the second coordinate. Thus, we suppose $S$ contains some $(f_1,x')$.
Since $S$ is minimal and contains $(f,e)$, we must have $f_1 \neq f$. 
Therefore $\langle (f_1,x'), (f,f') \rangle = G$, so combining the minimality of~$S$ with the definition of~$r_2$ tells us that
$r_2 = (e,f')$. Thus, \cref{elt-order-2p} applies, after interchanging the two factors of~$G$.

If the third element of~$S$ is of the form $(x,x')$, then $\langle S - \{r_2\}\rangle$ has index~$2$ in~$G$, so $\Cay(G;S)$ has a hamiltonian cycle by \pref{Sr1Index2}. Thus, from the preceding paragraph, we may assume that the third element of~$S$ is of the form
$(f_1,e)$, $(e,f_1')$, or $(f_1,f_1')$.
However, only the last of these can generate all of~$G$ when combined with $r_1$ and~$r_2$.
Therefore
	$$ S = \{ r_1, r_2, (f_1,f_1') \} .$$
Furthermore, we must have $f_1 \neq f$ and $f_1' \neq f'$.

If $r_2=(f ,f ')$, then $|r_2(f _1,f _1')|=p$ and $|\langle r_1, (f _1,f
_1')\rangle|
=4p$, so \cref{direct-prod-cor} applies.  If
$r_2=(e,f ')$
then \cref{direct-prod-cor} cannot be applied,
but we claim that
 $$\bigl(\bigl((f _1,f _1'),(f ,e)\bigr)^p \#,(e,f ')\bigr)^{2p}$$ is a
hamiltonian cycle.
To verify this, we first calculate that:
	\begin{itemize}
	\item $ \bigl( (f_1f)^p f^{-1},(f_1')^pf'  \bigr)= (f,f_1' f') $,
	\item $ \bigl( (f_1f)^i , (f_1')^i  \bigr) 
	\in \langle (f,f_1' f') \rangle \cdot
		\begin{cases}
		\bigl( (f_1 f)^i, e \bigr) & \text{if $i$ is even}, \\
		\bigl( (f_1 f)^i, f_1' \bigr) & \text{if $i$ is odd},
		\end{cases}$
	\item $ \bigl( (f_1f)^i , (f_1')^i  \bigr) \cdot (f_1,f_1')
	\in \langle (f,f_1' f') \rangle \cdot
		\begin{cases}
		\bigl( (f_1 f)^{-(i+1)}, f_1' \bigr) & \text{if $i$ is even}, \\
		\bigl( (f_1 f)^{-(i+1)}, e \bigr) & \text{if $i$ is odd}.
		\end{cases}$
	\end{itemize}
The conclusion that we have a hamiltonian cycle now follows easily from \cref{FGL(notnormal)}.

\begin{subsubcase}
Assume ${(f_1,f') \in S}$ with ${f_1 \neq f}$.
\end{subsubcase}

If the third
element of $S$ is of the form $(f _2, f _1')$, then $f
_1'\neq f '$ (otherwise $S$ would not generate~$G$).
Because $|(f _2,f _1')(f
_1,f ')|=p$
and $|\langle (f ,e),(f _1,f ')\rangle|=4p$, the conditions of \cref{direct-prod-cor} are satisfied with $s_1=(f _2,f _1')$.

If the third element of~$S$ is of the form $(f_2,x')$, then $x'$ must be nontrivial (or else $S$ would not generate~$G$), so $f_2=f _1$; otherwise, $(f ,e)$ is redundant. 
Then, since $|(f _1,x')(f ,e)|=p$ and $|\langle (f ,e),(f _1,f ')\rangle|=4p$,
the conditions of \cref{direct-prod-cor} are satisfied.

If the third element of~$S$ is of the form $(e,x')$, then, since $|x'| = p$, \cref{NormalEasy} provides a hamiltonian cycle in $\Cay(G;S)$.

Finally, since we are not in \cref{elt-order-2p} or \ref{abelian-order2-generator}, we may now assume the third element of~$S$ is of the form $(x,x')$, with $x$ and~$x'$ nontrivial. Let $s_1=(f ,e)$ and $s_2=(f _1,f ')$.
We have $|s_1s_2|=2p$, and $|\langle (x,x'),(f _1,f ')\rangle|=2p$.
Furthermore, $\langle s_1s_2\rangle =\langle (x,e), (e,f') \rangle$.
The intersection of this group with $\langle (x,x') , (f _1, f ')\rangle$
is clearly trivial, so \cref{Stud71} applies.

\begin{case}
Assume ${|S| = 4}$.
\end{case}
Because $S$ generates $G/(\integer_p \times \integer_p)$, we know that it contains an element of the form $(f,e)$ or $(e,f')$. Let us assume $(f,e) \in S$.

\begin{subcase}
Assume $S \cap (\integer_p \times \integer_p) \neq \emptyset$.
\end{subcase}
Let $(x,x') \in S \cap (\integer_p \times \integer_p)$. We may assume that $x$ and~$x'$ are both nontrivial, for otherwise $\langle (x,x') \rangle \normal G$, so \cref{NormalEasy} applies. But then $\langle (f,e), (x,x') \rangle$ has index~$2$ in~$G$, which contradicts the fact that $|S| = 4$.

\begin{subcase}
Assume $S \cap (\integer_p \times \integer_p) = \emptyset$.
\end{subcase}
We may assume $S$ contains an element of the form $(f_1, f')$ with $f_1$ and~$f'$ nontrivial, for otherwise $\Cay(G;S)$ is isomorphic to a Cartesian product $\Cay(D_{2p}; S_1) \times \Cay(D_{2p}; S_1')$, and the Cartesian product of hamiltonian graphs is hamiltonian.

We must have $f_1 = f$, for otherwise $\langle (f,e), (f_1,f') \rangle$ has index~$p$ in~$G$, contradicting the fact that $|S| = 4$. So $(f,f') \in S$.

In order to generate~$D_{2p} \times \{e\}$, $S$ must contain an element whose first coordinate is not~$f$. The second coordinate must be trivial (for otherwise combining it with $(f,e)$ generates a subgroup of index~$p$). That is, we have $(f_2,e) \in S$, with $f_2 \neq f$. But then $\langle (f_2,e), (f,f') \rangle$ generates a subgroup of index~$p$, contradicting the fact that $|S| = 4$.
\end{proof}

\begin{prop} \label{4p2}
   If\/ $|G| = 4p^2$, where $p$ is a prime, then every connected Cayley graph on $G$
   has a hamiltonian cycle.
\end{prop}

\begin{proof}
Let $S$ be a minimal generating set of~$G$.
   Clearly we can assume that $p \geq 3$, for otherwise $|G|$ is a prime power, so \cref{{pk}}
   applies.
   Let $P$ denote a Sylow $p$-subgroup of $G$, and let $P_2$ be a Sylow $2$-subgroup.

   \setcounter{case}{0}

   \begin{case} 
   Assume ${P \notnormal G}$.    
   \end{case}
   Then by Sylow's theorem \pref{NoMod1->Normal} we have $p = 3$ and $|G : N_G(P)| = 4$, so $N_G(P) = P$.
   Since $P$ is Abelian we thus have $P = Z \bigl( N_G(P) \bigr)$ and Burnside's theorem
   on normal $p$-complements \cite[Thm.~7.4.3]{Gorenstein-FinGrps} implies that $P_2 \normal G$.
   Since $[G,G] \leq P_2$ we can assume $[G,G] = P_2 \cong \integer_2 \times \integer_2$ for
   otherwise \cref{KeatingWitte} applies. The kernel of the
   action of $P$ on $\integer_2 \times \integer_2$ is thus of order $3$ and the only
   groups to consider are
   $\integer_9 \ltimes (\integer_2 \times \integer_2)$ and
   $A_4 \times \integer_3$, which are covered in \cref{Z3kx(Z2xZ2)} and \cref{A4xZp}, respectively.

   \begin{case} 
   Assume ${P \normal G}$.   
   \end{case}
   Since $P_2$ is Abelian we have $[G,G] \leq P$ and we can assume
   $[G,G] = P \cong \integer_p \times \integer_p$ for otherwise
   \cref{KeatingWitte} applies. 
   
   If $P_2 \cong \integer_2 \times \integer_2$ then
   either there is a nonidentity element of $P_2$ acting trivially on $P$, in which case
   $G$ is of dihedral type and \cref{DihType} applies, or all three nonidentity elements
   of $P_2$ act nontrivially on $P$ and $G \cong D_{2p} \times D_{2p}$ which is covered
   by \cref{D2pxD2p}. 
   
   We can thus assume that $P_2 \cong \integer_4$.
   Denote a generator of $\integer_4$ by $f$ and the generators of $\integer_p \times \integer_p$
   by $x$ and $y$. 
      Note that since $[G,G] = P$, we know $f$ cannot act trivially on any
   nonidentity element of $P$.

   Because $S$ generates $G/(\integer_p \times \integer_p)$, it has to contain an element of the form $fx^iy^j$ or its inverse. Conjugating
   by an appropriate power of $x$ and $y$, we can thus assume 
   	$$f \in S .$$
   Also, in order to generate~$P$, the set~$S$ must have at least one element of the form $f^ix^jy^k$, with either
   $j$ or $k$ nonzero.
 
   Also, since $(f^2)^2 = e$, we know that $1$ and~$-1$ are the only possible eigenvalues of the linear transformation defined by $f^2$ on the vector space $\integer_p \times \integer_p$. Thus, by choosing $x$ and~$y$ to be eigenvectors of~$f^2$, we may assume $x^{f^2} \in \{x^{\pm 1}\}$ and $y^{f^2} \in \{y^{\pm 1}\}$.

 \begin{subcase}
 Assume $f^2$ acts trivially on $\integer_p \times \integer_p$.
 \end{subcase}
   This means that $f$ acts by an automorphism of order~$2$. Since the automorphism does not fix any nontrivial element of $\integer_p \times \integer_p$, this implies that $g^f = g^{-1}$ for all $g \in \integer_p \times \integer_p$. Hence $G$ is of quaternion type, so \cref{QnHamCyc} applies.

 \begin{subcase}
 Assume $|S| = 2$ {\upshape(}and $f^2$ is nontrivial on $\integer_p \times \integer_p${\upshape)}.
 \end{subcase}
Write $S = \{\, f, f^k z \}$, with $z \in \langle x, y \rangle$. We may assume $0 \le k \le 2$. Note that we must have $z^f \notin \langle z \rangle$, because $S$ generates~$G$.
 
 \begin{subsubcase}
 Assume $k = 0$.
 \end{subsubcase}
 Note that, since $[z, f]$ is not in $\langle z \rangle$ or $\langle z \rangle^f$, 
       	\begin{itemize}
	\item every element of $\integer_p \times \integer_p$ has a unique representation of the form $[z, f]^i z^j$ with $0 \le i,j < p-1$,
	and
	\item every element of $f^{-1} \langle z \rangle$ has a unique representation of the form $[z, f]^i z^{-1}  f^{-1} z^{-j}$ with $0 \le i,j < p-1$.
	\end{itemize}
Therefore the vertices visited by the path $( z^{p-1}, f^{-1}, z^{-(p-1)}, f)^p\#$ are all in different right cosets of $\langle f^{-2} \rangle$, so \cref{FGL(notnormal)} tells us that 
	$$\bigl( ( z^{p-1}, f^{-1}, z^{-(p-1)}, f )^p\#, f^{-1} \bigr)^2 $$
is a hamiltonian cycle in $\Cay(G;S)$.
 
 \begin{subsubcase}
 Assume $k = 1$.
 \end{subsubcase} 
Note that
	$ \bigl( (f^3,fz)^{p-1},f^{-3} \bigr) $
is a hamiltonian path in the subgraph of $\Cay \bigl( G; \{f, fz\} \bigr)$ induced by $\langle z \rangle \langle  f \rangle$. Letting $g = z^{-1}( z^{-1})^{f^{-1}}$, we have $g \notin \langle z \rangle$, so the vertices in the path are all in different right cosets of $\langle g \rangle$. Therefore,  \cref{FGL(notnormal)} tells us that
	$$ \bigl(( f^3,fz \bigr)^{p-1}, f^{-3}, (fz)^{-1} \bigr)^p $$
is a hamiltonian cycle in $\Cay(G;S)$.
   
 \begin{subsubcase}
 Assume $k = 2$.
 \end{subsubcase} 
 
 \begin{subsubsubcase}
 Assume $f^2$ does not invert $\integer_p \times \integer_p$.
 \end{subsubsubcase}
We may assume $x^{f^2} = x^{-1}$, $y^{f^2} = y$, and $z = xy$.
\cref{DoubleEdge} applies, because
      $(f^2xy)^2 = y^2 \in \langle y \rangle \normal G$.

 \begin{subsubsubcase}
 Assume $f^2$ inverts $\integer_p \times \integer_p$.
 \end{subsubsubcase}
 We claim that
      $$ \Bigl(  \bigl( (f^2z,f^{-1})^2,(f^2z,f)^2 \bigr)^{(p-1)/2},f^2z,f^{-1},f^2z,f\Bigr) ^p$$
 is a hamiltonian cycle.

To see this, we first calculate the product
	$$  \Bigl( \bigl( (f^2z)(f^{-1}) \bigr)^2\bigl((f^2z)(f) \bigr)^2 \Bigr)^{(p-1)/2}(f^2z)(f^{-1})(f^2z)(f)
	= z^{-1} .$$
Now, we calculate the vertices  $(g_0,g_1,g_2,\ldots,g_{4p-1})$ visited by the walk 
	$$ \Bigl(  \bigl( (f^2z,f^{-1})^2,(f^2z,f)^2 \bigr)^{(p-1)/2},f^2z,f^{-1},f^2z\Bigr) $$
in the quotient graph $\langle z \rangle \backslash \Cay(G; S)$. Letting $y = z^f$, and noting that each vertex of the quotient has a unique representative in $\langle y \rangle \langle f \rangle$,
we calculate that if
	\begin{itemize}
	\item $0 \le i < (p-1)/2$ and $0 \le j < 8$, or
	\item $i = (p-1)/2$ and $0 \le j < 4$,
	\end{itemize}
then
	$$
	g_{8i+j} =
	\langle z \rangle \, y^{2i} \cdot
	\begin{cases}
	\hfil e & \text{if $j = 0$}, \\
	\hfil f^2 & \text{if $j = 1$}, \\
	\hfil f & \text{if $j = 2$}, \\
	\hfil yf^3 & \text{if $j = 3$}, \\
	\hfil yf^2& \text{if $j = 4$}, \\
	\hfil y & \text{if $j = 5$}, \\
	\hfil yf & \text{if $j = 6$}, \\
	\hfil y^2 f^3 & \text{if $j = 7$}.
	\end{cases} $$
All of these are distinct, so \cref{FGL(notnormal)} tells us we have a hamiltonian cycle.

 \begin{subcase}
 Assume $|S| = 3$ {\upshape(}and $f^2$ is nontrivial on $\integer_p \times \integer_p${\upshape)}.
 \end{subcase}

   \begin{subsubcase} 
   Assume ${x^{f^2} = x^{-1}}$ and ${y^{f^2} = y}$.
   \end{subsubcase}      
   Since $f$ obviously commutes with $f^2$, it must preserve the eigenspaces of~$f^2$. So $x^f \in \langle x \rangle$ and $y^f \in \langle y \rangle$. So $x^f = x^r$, where $r^2 \equiv -1 \pmod{p}$, and $y^f = y^{-1}$.
   Since $|S| = 3$, we know that $S$ does not contain any element of the form ${f^\ell x^iy^j}$ with $i$ and $j$ both nonzero, so we can clearly assume that $S = \{f, f^{\ell_1}x, f^{\ell_2}y\}$. We may also assume $0 \le \ell_1,\ell_2 \le 2$ (after replacing generators by their inverses, if necessary).
      \begin{itemize}
      \item If $\ell_2 = 0$, then \cref{NormalEasy} applies
      since $\langle y \rangle \normal G$. 
      \item If $\ell_2 = 1$, then \cref{DoubleEdge} applies
      since $fy \equiv f \pmod{\langle y \rangle}$ and $\langle y \rangle \normal G$. 
       \item If $\ell_2 = 2$, then \cref{DoubleEdge} applies
      since $(f^2 y)^2 = y^2$ and $\langle y^2 \rangle = \langle y \rangle \normal G$. 
      \end{itemize}
      
   \begin{subsubcase} 
   Assume ${x^{f^2} = x^{-1}}$ and ${y^{f^2} = y^{-1}}$.
   \end{subsubcase}     
We may assume $S = \{f, f^{\ell_1}x, f^{\ell_2}y\}$. Since $|S| = 3$, we must have $x^f \in \langle x \rangle$ and $y^t \in \langle y \rangle$. Therefore $x^f = x^a$ and $y^f = y^b$, where $a$ and~$b$ are square roots of~$-1$ in the field $\integer_p$.

	 If either $\ell_1$ or~$\ell_2$ is equal to $0$, $1$, or~$3$, then we may
   apply \cref{NormalEasy} or \cref{DoubleEdge}, because $\langle x \rangle$ and $\langle y \rangle$ are both normal subgroups of~$G$. We are therefore left with
   $S = \{f, f^2x, f^2y\}$. Take $s_1 = f^2x$ and $s_2 = f^2y$. Then $s_1s_2 = x^{-1} y$ is of order~$p$
   and since $|\langle f, f^2y\rangle| = 4p$ and $s_1s_2 \notin \langle f, f^2y\rangle$, we see that \cref{Stud71}
   applies.
\end{proof}


\section{Groups of order $pqr$} \label{pqrSect}

\begin{prop} \label{2pq}
If\/ $|G| = 2pq$, where $p$ and~$q$ are prime, then every connected Cayley graph on~$G$ has a hamiltonian cycle.
\end{prop}

\begin{proof}
Let $S$ be a minimal generating set of~$G$.

We may assume $p$ and~$q$ are distinct, for otherwise $|G| = 2p^2$, so \cref{2p2} applies.
Thus, there is no harm in assuming that $p < q$. We may also assume that $p,q \ge 3$, for otherwise $|G|$ is of the form $4p$, so \cref{4p} applies.

Let $Q$ be a Sylow $q$-subgroup of~$G$.
Because $|G| = 2pq = 2 \times \text{odd}$, it is well known that $G$ has
a (unique) normal subgroup of order~$pq$ \cite[Thm.~4.6]{Wielandt-FinitePermGrps}. Since
$p < q$, Sylow's Theorem \pref{NoMod1->Normal} tells us that $Q$ is normal in this subgroup.
Then, being characteristic, $Q$ is normal in~$G$.

The quotient group $G/Q$ is of order~$2p$. We may assume it is
nonabelian, for otherwise $[G,G] =
Q$ is cyclic of prime order, so \cref{KeatingWitte} applies. 
Therefore $G \iso D_{2p} \ltimes Q$. Because $\Aut Q
\iso (\integer_q)^\times$ is abelian, we know that the commutator subgroup
of~$D_{2p}$ centralizes~$Q$. Hence $G \iso \integer_2 \ltimes(\integer_p
\times \integer_q)$. Then either $G \iso D_{2pq}$ is dihedral (so
\cref{dihedral} applies) or $[G,G]$ has prime order (so
\cref{KeatingWitte} applies). 
\end{proof}

\begin{prop}[D.\,Li \cite{Li-pqr}] \label{pqr}
 If $|G| = pqr$, where $p$, $q$, and~$r$ are distinct primes, then every
connected Cayley graph on~$G$ has a hamiltonian cycle.
 \end{prop}

\begin{proof}
The case where $|G| = 2pq$ has been discussed in \cref{2pq}, so let us assume $|G|$ is odd.

Also assume $p$ is the smallest of $p$, $q$ and~$r$. 
  Because $|G|$ is square-free, it is well known (and not difficult to
prove) that $[G,G]$ must be cyclic of some order dividing $|G|/p$ \cite[Cor.~9.4.1]{Hall-ThyGrps}.
(In particular, $G$ is solvable.)
We may assume $|[G,G]|$ is not prime, so we conclude that
 $$ [G,G] \iso \integer_{qr} .$$
 Thus, $G$ is a semidirect product: up
to
isomorphism, we have
 $$ G = \integer_p \ltimes ( \integer_q \times \integer_r ) ,$$
 where $\integer_p$ acts nontrivially on both $\integer_q$
and~$\integer_r$.

Now, let $S$ be a minimal generating set of~$G$.
Since $(\integer_q \times \integer_r) \cap Z(G) = \emptyset$, \cref{NormalEasy} tells us that we may assume $S \cap (\integer_q \times \integer_r) = \emptyset$. Thus, every element of~$S$ has order~$p$.

\setcounter{case}{0}

\begin{case}
Assume ${|S| = 2}$.
\end{case}
 We may write $S = \{s,t\}$. 
 We have $t = x s^k$ for some generator~$x$ of $\integer_q \times
\integer_r)$ and some integer~$k$ with $1 \le k < p$. Then
 $$ t s^{-(k-1)} t s^{p-k-1} = x s x s^{-1} $$
 is a generator of $\integer_q \times \integer_r$ (because $s$, being of
odd order, cannot invert either $\integer_q$ or $\integer_r$). So \cref{FGL} tells us that
 $$ ( t, s^{-(k-1)}, t, s^{p-k-1})^{qr} $$
 is a hamiltonian cycle in $\Cay(G;S)$.

\begin{case}
Assume ${|S| = 3}$.
\end{case}
 We may write $S = \{s,t,u\}$. 
 The minimality of~$S$, together with the fact that $S \cap (\integer_q \times \integer_r) = \emptyset$, implies $t = s^i x$ and $u = s^j y$, where $x$ and~$y$ are generators of $\integer_q$ and~$\integer_r$, respectively, and $1 \le i,j < p$. Then $\langle t,u \rangle = G$,
 which contradicts the
fact that the generating set~$S$ is minimal.
 \end{proof}

\begin{cor} \label{3pq}
 If\/ $|G| = 3pq$, where $p$ and $q$ are prime, then every connected Cayley graph
on~$G$ has a hamiltonian cycle.
 \end{cor}

\begin{proof}
Note that:
	\begin{itemize}
	\item We may assume $p, q \ge 3$, for otherwise $|G|$ is of the form $2pq$ or $2p^2$ so \cref{2pq} or \cref{2p2} applies.
	\item We may assume $p \neq q$, for otherwise $|G| = 3p^2$, so \cref{3p2} applies. 
	\item We may assume $p,q > 3$, for otherwise $|G|$ is of the form $p^2 q$ with $p \not\equiv 1 \pmod{q}$, so \cref{p2q-easy} applies.
	\end{itemize}
Thus, $|G|$ is the product of three distinct primes, so \cref{pqr} applies.
\end{proof}

\begin{cor} \label{5pq}
 If\/ $|G| = 5pq$, where $p$ and $q$ are distinct primes, then every connected Cayley graph
on~$G$ has a hamiltonian cycle.
 \end{cor}

\begin{proof}
Note that:
	\begin{itemize}
	\item We may assume $p, q \ge 5$, for otherwise $|G|$ is of the form $2pq$ or $2p^2$ or $3pq$ or $3p^2$, so \cref{2pq} or \cref{2p2} or \cref{3pq} or \cref{3p2} applies.
	\item Then we may assume $p,q \neq 5$, for otherwise $|G|$ is of the form $p^3$ or $p^2 q$ with $p^2 \not\equiv 1 \pmod{q}$, so \cref{pk} or \cref{p2q-easy} applies.
	\end{itemize}
Thus, $|G|$ is the product of three distinct primes, so \cref{pqr} applies.
\end{proof}


\section{Groups of order $4pq$} \label{4pqSect}

We start by considering a special case.

\begin{prop} \label{4pq(semiprod)}
If $G = P_2 \ltimes \integer_{pq}$ is a semidirect product of a group~$P_2$ of order~$4$ and a cyclic group\/ $\integer_{pq}$ of order $pq$, where $p$ and~$q$ are distinct odd primes, then every connected Cayley graph on~$G$ has a hamiltonian cycle.
\end{prop}

\begin{proof}
Let $S$ be a minimal generating set of~$G$.
We may assume $[G,G] = \integer_{pq}$ (otherwise \cref{KeatingWitte}
applies).

\setcounter{case}{0}

\begin{case}
 Assume ${P_2 \iso \integer_4}$, so ${G \iso \integer_4 \ltimes
\integer_{pq}}$.
 \end{case}
 We can view $G$ as $\integer_4 \ltimes (\integer_p \times
\integer_q)$. Let $f$ be a
generator of $\integer_4$ and let $x$ and $y$ be generators of
$\integer_p$ and
$\integer_q$, respectively. There exists $r \in \integer$, such that $x^f = x^r$ and $y^f = y^r$. We have $r^4 \equiv 1 \pmod{pq}$, because $|f| = 4$. 

Note that in view of \cref{NormalEasy} and
the fact that $Z(G) \cap (\integer_{p} \times \integer_q) =
\{e\}$, we can assume that no
element of the form $x^j y^k$ is in $S$.

Since $S$ clearly contains at least one element of the form $f^i x^j y^k$, where
$i \in \{1,3\}$, we can assume $f \in S$.

\begin{subcase}
 Assume the action of\/ ${\integer_4}$ on ${\integer_{pq}}$ is \textbf{not} faithful.
 \end{subcase}
 Then $r^2
\equiv 1 \pmod{pq}$. Since $[G,G] = \integer_{pq}$, this implies $r \equiv
-1
\pmod{pq}$ and thus $f$ inverts every element of $\integer_p \times
\integer_q$. But then $G \iso Q_{4pq}$ is of quaternion type, so \cref{QnHamCyc} applies.

\begin{subcase}
Assume the action of\/ ${\integer_4}$ on ${\integer_{pq}}$ is faithful.
 \end{subcase}
 This means $r^2 \not\equiv 1 \pmod{pq}$, so we may assume $r^2 \not\equiv 1 \pmod{p}$ (by interchanging $p$ and~$q$ if necessary). Therefore $r^2 \equiv -1 \pmod{p}$.
 
 \begin{subsubcase}
 Assume $S$ contains an element~$s$ of the form $fx^jy^k$ or $f^2x^jy^k$, where both $j$ and $k$ are nonzero. 
 \end{subsubcase}
Now if $s = fx^jy^k$, then clearly $f^3s = x^jy^k$ is of
order $pq$ and thus  \cref{FGL} tells us that $\bigl( f^3 ,s \bigr)^{pq}$ is a hamiltonian cycle in
$\Cay(G;S)$.
On the other hand, if $s = f^2x^jy^k$, then
	$$f^{-1} s^{-1} f s = f ^{-1} (x^j y^k)^{-1} f (x^j y^k) = [f, x^j y^k] $$
generates $[G,G] = \integer_{pq}$. Thus, it is easy to see from \cref{FGL} that  $(f,s,f^{-1},s^{-1})^{pq}$ is a hamiltonian cycle in
$\Cay(G;S)$.

 \begin{subsubcase}
 Assume all elements of $S - \{f\}$ are of the form $f^ix^jy^k$ where $i \in \{1,2\}$ and
precisely one of $j,k$ is nonzero. 
\end{subsubcase}
Clearly, $S = \{f, f^i x, f^j y\}$ (perhaps after replacing $x$ and~$y$ by their powers).
 Since $(fx)^{-1} fy = x^{-1}y$, we see
that $\langle fx, fy \rangle = G$. Similarly, 
$[fx , f^2 y] = x^{-2} [f,y]$ generates $[G,G]$,
and thus $\langle fx, f^2y \rangle = G$. 
Since $S$ is minimal, there
are
thus only two possibilities for~$S$. First if $s_1 = f^2x$, $s_2 =
fy \in S$, then $s_1 f^{-1} s_1 s_2 = x^{r-1}y$ is of order $pq$
and thus \cref{FGL} tells us $(s_1,f^{-1},s_1,s_2)^{pq}$ is a hamiltonian cycle in $\Cay(G ;
S)$. Finally, if $s_1 = f^2x$ and $s_2 = f^2y$, then $fs_1f^{-1}s_2 =
x^r y$ is of order $pq$, so  \cref{FGL} tells us $(f,s_1,f^{-1},s_2)^{pq}$ is a hamiltonian
cycle in $\Cay(G ; S)$.

\begin{case}
 Assume ${P_2 \iso \integer_2 \times \integer_2}$, so ${G \iso (\integer_2
\times \integer_2) \ltimes \integer_{pq}}$.
 \end{case}

\begin{subcase}
 Assume some involution in ${\integer_2 \times \integer_2}$ centralizes
${\integer_{pq}}$. {\upshape(}That is, ${\integer_2 \times \integer_2}$ is not
faithful
on ${\integer_{pq}}$.{\upshape)}
 \end{subcase}
 Then $G \iso \integer_2 \ltimes \integer_{2pq} \iso D_{4pq}$ is
dihedral,
so \cref{dihedral} applies.

\begin{subcase}
 Assume no involution in ${\integer_2 \times \integer_2}$ centralizes
${\integer_{pq}}$. {\upshape(}That is, ${\integer_2 \times \integer_2}$ is faithful
on ${\integer_{pq}}$.{\upshape)}
 \end{subcase}
 Then $G \iso D_{2p} \times D_{2q}$.

If $S$ contains an element~$s$ of $\integer_{pq}$, then $\langle s \rangle \normal G$. 
Also, since $[G,G] = \integer_p \times \integer_q$, we see that $Z(G)
\cap(\integer_p \times \integer_q) = \{e\}$. Thus \cref{NormalEasy} applies.

Henceforth, we assume $S \cap \integer_{pq} = \emptyset$.
 Also, we consider $G$ to be $D_{2p} \times D_{2q}$.

Let us consider the possibility that $fx' \in S$, where $f$ is a reflection in~$D_{2p}$, and $x'$ is a nontrivial rotation in~$D_{2q}$ (so $x'$ generates~$\integer_q$). We may assume that
$S
- \{fx'\}$ generates $G/\integer_q$ (otherwise,
\cref{DoubleEdge} applies with $N = \integer_q$). Furthermore, it is easy to see that the only proper subgroups
of~$G$ that properly contain $fx'$ are $D_{2p} \times \integer_{q}$ and
$\langle
f \rangle \times D_{2q}$. It is therefore not difficult to see that (up to isomorphism) the
only possible Cayley graph is:
	$$ \text{$\Cay \bigl( D_{2p} \times D_{2q}; \{fx', ff', fx \} \bigr)$ where $f,x \in D_{2p}$ and $f',x' \in D_{2q}$} .$$
Note that $(fx, fx')$ is obviously a hamiltonian cycle in 
	$$ \Cay \left( \frac{\langle f, x, x' \rangle }{ \langle x, x' \rangle} ; \{ fx, fx' \} \right) \iso \Cay \bigl( \langle f \rangle; \{f\} \bigr) ,$$
 so, since $(fx) (fx') = x^{-1}x'$ generates $\langle x, x' \rangle$, \cref{FGL} implies that 
	$$  \bigl( ( fx,fx' )^{pq}\# \bigr) $$
is a hamiltonian path in $\Cay \bigl( \langle f, x, x' \rangle ; \{fx, fx'\} \bigr)$. Therefore, all of the vertices of this path are in different right cosets of $\langle f' x' \rangle$. So \cref{FGL(notnormal)} tells us that
	$$ \bigl( ( fx,fx' )^{pq}\#, ff' \bigr)^2 $$ 
is a hamiltonian cycle in $\Cay(G;S)$.

 We may now assume there is no double edge in $\Cay(G/\integer_p; S)$ or
$\Cay(G/\integer_q;S)$. (That is, if $s \in S$, and $s$ represents an element of order~$2$
in $G/\integer_p$ or  $G/\integer_q$, then $s$ has order~$2$ in~$G$.)
 This implies that $S$ consists entirely of
elements of order~$2$.

If $S$ has four (or more) elements, it is clear from the minimality of~$S$ that $S$ has the form $S = \{ f, fx, f', f'x' \}$ (with $\langle f,x \rangle = D_{2p} \times \{e\}$ and $\langle f',x' \rangle = \{e\} \times D_{2q}$. This means that $\Cay(G;S)$ is isomorphic to the Cartesian product $C_{2p} \times C_{2q}$ of two cycles, which obviously has a hamiltonian cycle \cite[Cor.\ on p.~29]{ChenQuimpo-hamconn}.

We now assume that $S = \{s,t,u\}$ consists of precisely three elements of order~$2$.

\begin{subsubcase}
 Assume ${S \cap D_{2p} = \emptyset}$ and ${S \cap D_{2q} = \emptyset}$.
 \end{subsubcase}
 Then $s \equiv t \equiv u \pmod{\integer_{pq}}$. This is impossible,
because $G/\integer_{pq} \iso \integer_2 \times \integer_2$ is not
cyclic.

\begin{subsubcase}
 Assume ${s \in S \cap D_{2p}}$ and ${t \in S \cap D_{2q}}$.
 \end{subsubcase}
 Then $s \in Z(G/\integer_p)$ and $G/\integer_p \iso D_{4q}$. Since
$\{ t,u \}$ generates $G/\integer_p$, and $\Cay \bigl( G/\integer_p; \{t,u\} \bigr)$ is a $4q$-cycle,
we see that $\Cay(G/\integer_p;S)$
is isomorphic to $\Cay \bigl( \integer_{4q}; \{1,2q\} \bigr)$. So \cref{VoltageCor} implies that some
hamiltonian cycle in $\Cay(G/\integer_p;S)$ lifts to a hamiltonian cycle
in $\Cay(G;S)$.

\begin{subsubcase}
 Assume ${s \in S \cap D_{2p}}$ and ${S \cap D_{2q} = \emptyset}$.
 \end{subsubcase}
Each element of $G = D_{2p} \times D_{2q}$ is an ordered pair. 
Also, since $\{t,u\}$ generates $G/D_{2p}$, we know that neither $t$ nor~$u$ belongs to $D_{2p}$ (and, by assumption, they do not belong to $D_{2q}$), so
$t \equiv u \pmod{\integer_{pq}}$. Therefore, we may
write 
	$$ \text{$t = (f,f')$, $s = (f x, e)$, $u = (f x^k, f' x')$,}$$
where $f$ and $f'$ are reflections, $x$ and $x'$ are nontrivial rotations,
and $k \in \integer$.

%

\begin{subsubsubcase}
Assume $k \not\equiv 2 \pmod{p}$.
\end{subsubsubcase}
Since $(s,t,s,u)$ is a hamiltonian cycle in $\Cay \bigl( G/(\integer_p \times \integer_q); S \bigr)$, and
	$$ stsu
	= (fx \cdot f \cdot fx \cdot fx^k, \ e \cdot f' \cdot e \cdot f' x' )
	= \bigl( x^{k - 2}, x' \bigr) $$
generates $\integer_p \times \integer_q$ (because $k \not\equiv 2 \pmod{p}$), \cref{FGL} implies that $(s,t,s,u)^{pq}$ is a hamiltonian cycle in $\Cay(G;S)$.

\begin{subsubsubcase}
Assume $k\equiv 2 \pmod{p}$.
\end{subsubsubcase}
We claim that $\Cay(G;S)$ is a generalized Petersen graph.
To see this, begin by letting 
	$$ \text{$x_{2i} = (ut)^i$ and $x_{2i+1} = (ut)^i u$ for $0 \le i < pq$}$$
and
	$$ \text{$y_j = sx_j$ for $0 \le j < 2pq$} .$$
Then every vertex of $\Cay \bigl( G; \{t,u\} \bigr)$ is in the union of the two disjoint $2pq$-cycles 
	$$ \text{$(x_0,x_1,x_2,\ldots,x_{2pq-1}, x_{2pq})$ and $(y_0,y_1,y_2,\ldots,y_{2pq-1}, y_{2pq})$.} $$
Now, write $(x^2,e) = (ut)^r$ with $1 \le r < pq$. Then
	$$ \text{$sts = (fx^2, f') =  t(ut)^r$ and $sus = (fx^{2-k}, f'x')  = (f, f'x') = (ut)^r u$} $$
so, by induction on~$i$, we see that 
	$$x_i s = s \, x_{(2r+1)i} = y_{(2r+1)i} ,$$
which means there is an $s$-edge from $x_i$ to $y_{(2r+1)i}$. Therefore, $\Cay \bigl(G; \{s,t,u\} \bigr)$ is a generalized Petersen graph, as claimed.

Now, if we let $n = 2pq$, then: 
	\begin{itemize}
	\item The number of vertices of $\Cay(G;S)$ is $2n$.
	\item Since $pq$ is odd, we know that $n = 2pq \not\equiv 0 \pmod{4}$.
	\item Since $2pq$ is even, we know $n = 2pq \not\equiv 5 \pmod{6}$.
	\end{itemize}
Therefore \cref{GenPet} tells us that $\Cay(G;S)$ has a hamiltonian cycle.
\end{proof}

\begin{prop} \label{4pq}
If\/ $|G| = 4pq$, where $p$ and~$q$ are prime, then every connected Cayley graph on~$G$ has a hamiltonian cycle.
\end{prop}

\begin{proof}
Let $S$ be a minimal generating set of~$G$.
We may assume $p$ and~$q$ are distinct (for otherwise $|G| = 4p^2$, so \cref{4p2} applies).
Furthermore, we may assume $p,q \ge 3$, for otherwise $|G|$ is of the form $8p$, so \cref{8p} applies. 
We also assume $G \not\iso A_5$ (since Cayley graphs on that group are covered in \cref{A5}). It then follows easily from Burnside's Theorem on normal $p$-complements that $G$ is solvable.

Let $P_2$ be a Sylow $2$-subgroup of~$G$. Because $G$ is solvable,  \cref{Hall-solv} tells us there are Hall subgroups~$H_{pq}$ and $H_{4q}$ of order~$pq$ and $4q$, respectively. 

There is no harm in assuming that $p > q$ (so $p \ge 5$). This implies the Sylow $p$-subgroup $\integer_p$ is normal in $H_{pq}$. 
So
 $$ |G : N_G(\integer_p)| \le |G : H_{pq}| = 4 < p + 1 .$$
Therefore $\integer_p \normal G$. So $G = H_{4q} \ltimes \integer_p$.

\setcounter{case}{0}
\begin{case}
Assume ${H_{4q}}$ has a normal Sylow $q$-subgroup.
\end{case}
We may assume
$H_{4q}$ is not abelian (otherwise \cref{KeatingWitte} applies). This implies
the commutator subgroup of $H_{4q}$ must be a Sylow
$q$-subgroup~$\integer_q$. Because $\Aut(\integer_p)$ is abelian, this
implies that $\integer_q$ centralizes~$\integer_p$. So $G = P_2 \ltimes
(\integer_p \times \integer_q) \iso P_2 \ltimes \integer_{pq}$. Therefore \cref{4pq(semiprod)} applies.

\begin{case}
Assume the Sylow $q$-subgroups of $H_{4q}$ are not normal.
\end{case}
Since $H_{4q}$ is of order~$4q$, Sylow's Theorem \pref{NoMod1->Normal} tells us there is a divisor of~$4$ that is congruent to~$1$ modulo~$q$. Clearly, we must have $q = 3$. Thus, $H_{4q}$ is a group of order $4 \cdot 3 = 12$, in which the Sylow $3$-subgroups are not normal. The only such group is~$A_4$. So $G \iso A_4 \ltimes \integer_p$.

 We have $G \iso \integer_3 \ltimes \bigl( (\integer_2 \times \integer_2)
\times \integer_p \bigr)$.
 We may assume $S \cap \integer_p = \emptyset$, for otherwise
\cref{NormalEasy} applies.

 We let $f$ be a generator of
$\integer_3$, $x$ and $y$ be generators of $\integer_2 \times
\integer_2$ and $z$ be a generator of $\integer_p$ where $x^f = xy$,
$y^f = x$ and $z^f = z^r$ for some $r \in \integer$ such that $r^3
\equiv 1\pmod{p}$.

\begin{subcase}
 Assume ${\integer_3}$ acts nontrivially on ${\integer_p}$.
 \end{subcase}
 Since $S$ must contain an element of $G - \langle x,y,z\rangle$, and all of these elements have order~$3$,  we can assume $f \in S$.

 \begin{subsubcase}
 Assume ${|S| = 2}$.
 \end{subsubcase}
 Let $S = \{f, s\}$.

\begin{subsubsubcase}
Assume ${s \in (\integer_2 \times \integer_2) \times \integer_p}$.
\end{subsubsubcase}
The generator $s$ gives a double edge in $G/\integer_p$ (and $\{f,s\}$ is a
minimal generating set of $G/ \integer_p$), so \cref{DoubleEdge}
applies.

 \begin{subsubsubcase}
Assume ${s \in f \bigl( (\integer_2 \times \integer_2) \times \integer_p
\bigr)}$.
\end{subsubsubcase}
 Write $s = fx^iy^jz^k$. 
 Since $S$ generates~$G$, we cannot have $i = j = 0$ or $k = 0$,
 so we can assume $s = fxz$. 

We show that $f$ and $s$ satisfy the conditions of
\cref{Stud61}. We have 
	$$ [f,s] = f^{-1}(fxz)^{-1} f (fxz) = f^{-1} z^{-1} x^{-1} f^{-1} f^2 xz = y z^{1-r} .$$
Since $f$ acts nontrivially on~$\integer_p$, we know $r \not\equiv 1 \pmod{p}$, so
$|[f,s]| = 2p$.  Also, we have
	$$ s = fxz \notin  \langle f \rangle \langle [f,s] \rangle .$$
The other two conditions are clearly satisfied.

\begin{subsubcase}
 Assume ${|S| = 3}$.
 \end{subsubcase}
 We may assume $S = \{f, f^i x, f^j z\}$, and $i,j \in \{0,1\}$.
 Since $S \cap \integer_p = \emptyset$, we have
 $j = 1$, for otherwise \cref{NormalEasy} applies.
But then $\{f, f z\}$ gives a double edge in $\Cay(G/\integer_p;S)$
(and $S - \{f z\}$ is a minimal generating set of
$G/\integer_p$), so \cref{DoubleEdge} applies.

\begin{subcase}
 Assume ${\integer_3}$ centralizes ${\integer_p}$.
 \end{subcase}
 This means $r = 1$, and we have $G \iso A_4 \times \integer_p$.
 Recalling that $S \cap \integer_p = \emptyset$, we know that no element of~$S$ has order~$p$.

Let $s$ be an element of~$S$ whose order is divisible by~$p$. 
Note that $\langle s \rangle$ contains a nontrivial subgroup
of $G/\integer_p \iso \integer_3 \ltimes (\integer_2 \times \integer_2)$. Either this subgroup is maximal (of
order~$3$) or we have $\langle s, t \rangle = G$ for any $t \in S$ with
$3 \mid |t|$. Therefore $|S| = 2$, so we may write $S = \{s,t\}$.

We may assume $s \notin (\integer_2 \times \integer_2) \times \integer_p$, for otherwise $\{s, s^{-1}\}$ gives a double edge in $\Cay(G/\integer_p;S)$, so \cref{DoubleEdge} applies. Therefore, we may assume $s = fz$.

We  show that $fz$ and $t$ satisfy the conditions of \cref{Stud61}. 
Since $\langle fz,t \rangle = G$, we may assume $t = f^\ell x z^k$ for some $\ell,k \in \integer$.
Then 
	$$ [fz,t] = [fz,f^\ell x z^k] = y .$$
Since all nonidentity elements of $\langle fz \rangle$ are in $\langle f \rangle \langle z \rangle$, 
we see that $t \notin \langle fz \rangle \langle
[s,t] \rangle$, and the remaining conditions are also clearly satisfied.
 \end{proof}


\section{Groups of order $2p^3$} \label{2p3Sect}

\begin{prop}  \label{2p3}
 If\/ $|G| = 2 p^3$, where $p$ is prime, then every connected Cayley graph on~$G$
has
a hamiltonian cycle.
 \end{prop}

\begin{proof}
Let $S$ be a minimal generating set of~$G$.

 We may assume $p \ge 3$. (Otherwise, $|G| = 2^4$ is a prime power, so \cref{pk} applies.)  Let $P$ be a Sylow $p$-subgroup of~$G$, so $G =
\integer_2 \ltimes P$, and let $f$ be a generator of~$\integer_2$.

We may assume $S \not\subset fP$ (for otherwise \cref{pkSubgrp}
applies). Thus, there exists $s \in S \cap P$.

\setcounter{case}{0}

\begin{case} \label{2p3-sNormal}
Assume $\langle s \rangle \normal G$. 
\end{case}
Note that if $|s| = p$, then \cref{NormalEasy} applies. Also, if $|s| = p^3$, then $\langle s \rangle = P \supset [G,G]$, so $[G,G]$ is cyclic of $p$-power order, so \cref{KeatingWitte} applies.
Thus, we may assume $|s| = p^2$. 

Also, we may assume $\langle s \rangle \cap Z(G)$ is nontrivial (else \cref{NormalEasy} applies), so it is clear that $f$ does not invert $\langle s \rangle$. Since $|f| = 2$, we conclude that $f$ centralizes~$s$. Since we may assume that $|[G,G]| \neq p$ (else \cref{KeatingWitte} applies), this implies that we may assume $P$ is nonabelian.

Now, for any $x \in P$, we have $\langle [x,s] \rangle \subset \langle s \rangle$ (because $\langle s \rangle \normal G$), so $f$ centralizes $[s,x]$. Therefore
	$$[x^f, s] = [x^f, s^f] = [x,s]^f = [x,s] ,$$
so $f$ centralizes~$x$, modulo $C_P(s) = \langle s \rangle$. Thus, $f$ centralizes both $P/\langle s \rangle$ and $\langle s \rangle$. Since $|f| = 2$ is relatively prime to~$|P|$, this implies that $f$ centralizes~$P$ (see \cite[Thm.~5.3.2]{Gorenstein-FinGrps}). Therefore $G = \integer_2 \times P$, so $[G,G] = [P,P]$ is cyclic of order~$p$, so \cref{KeatingWitte} applies.

\begin{case}
Assume $\langle s \rangle \notnormal G$. 
\end{case}
 There is an element~$a$ of~$S$ with $|a|$ even. 

\begin{subcase}
 Assume ${G = \integer_2 \times P}$.
 \end{subcase}
 Then $[G,G]$ is cyclic of order $1$ or~$p$, so \cref{KeatingWitte}
applies.

\begin{subcase}
 Assume ${P}$ is abelian {\upshape(}but ${G}$ is nonabelian{\upshape)}.
 \end{subcase}
 We may assume $G$ is not of dihedral type (else \cref{DihType}
applies). So $|[G,G]| \le p^2$. We may also assume $[G,G]$ is not cyclic (for
otherwise \cref{KeatingWitte} applies). Therefore $P =
(\integer_p)^3$, and 
	$$G = \bigl( \integer_2 \ltimes (\integer_p \times
\integer_p) \bigr) \times \integer_p
= \left\langle f, x, y, z 
\mathrel{\Bigg|}
\begin{matrix}
f^2 = x^p = y^p = z^p = e, \\
 x^f = x^{-1}, \ y^f = y^{-1}, \ z^f = z, \\
  \text{$\langle x, y, z\rangle$ is abelian}
 \end{matrix} \right\rangle
$$
is the direct product of a group of
dihedral type with a cyclic group of order~$p$.
Also note that, because $a^2$
is in the elementary abelian group~$P$, we have $|a^2| \in \{1,p\}$.

Since any two elements of order~$2$ always generate a dihedral group, it is easy to see that $G/Z(G)$ has no $2$-element generating set. Therefore $|S| \ge 3$.

\begin{subsubcase} \label{2p3-a2note}
Assume $a^2 \neq e$. 
\end{subsubcase}
We know $a^2$ is in $Z(G)$ (because it is centralized by both~$a$ and the abelian group~$P$), so we conclude that $\langle a^2 \rangle = Z(G)$ is normal in~$G$. 
Also, since $\langle a \rangle$ has index $p^2$ in~$G$, we know that $|S| \le 3$.
Then, since $G/Z(G)$ has no $2$-element generating set, we conclude that $S$ is a minimal generating set of $G/\langle a^2\rangle$. Thus, \cref{DoubleEdge} (with $N = \langle a^2 \rangle$ and $s = a = t^{-1}$) provides a hamiltonian cycle in $\Cay(G;S)$.

\begin{subsubcase}
Assume $a^2 = e$. 
\end{subsubcase}
We may assume $f \in S$. Since $\langle s \rangle \notnormal G$, we must have $|\langle f , s \rangle| > 2p$. Therefore, the minimality of~$S$ implies $|S| \le 3$. Since we already have the opposite inequality, we conclude that $|S| = 3$; write $S = \{f, s, t\}$.
 \begin{itemize}
 \item Suppose $t \notin P$. Then we may assume $t^2 = e$ (otherwise \cref{2p3-a2note} applies), so $ft$ is inverted by~$f$,
so it generates a normal subgroup of~$G$. 
Since $f \equiv t \pmod{\langle ft \rangle}$, the multigraph $\Cay \bigl( G/ \langle ft \rangle ; S \bigr)$ has double edges, and it is clear that all of its hamiltonian cycles use at least one of these double edges (since $\langle ft, s \rangle \neq G$). Therefore \cref{MultiDouble} applies.
 
  \item Suppose $t \in P$ (and $\langle t \rangle \notnormal G$, so \cref{2p3-sNormal} does not apply). 
  We may assume $s = xz$ and $t = y z^k$ for some $k \not\equiv 0 \pmod{p}$ (because $\langle s \rangle$ and $\langle t \rangle$ are not normal).
  We have $(t^{-1}f)^2 = z^{-2k} \in Z(G)$. Since $\langle s, t \rangle \cap Z(G) = \{e\}$, it is therefore clear that $\langle s,t \rangle \cap \langle t^{-1}f \rangle = \{e\}$, so all of the elements of $\langle s, t \rangle$ are in different right cosets of $\langle t^{-1}f \rangle$. Since $(s^{p-1},t)^p\#$ is a hamiltonian path in $\Cay \bigl( \langle s,t \rangle; \{s,t\} \bigr)$, this implies that all of the vertices in this path are in different right cosets of $\langle t^{-1}f \rangle$.

Then, since $|t^{-1} f| = 2p = |G|/| \langle s,t \rangle|$, \cref{FGL(notnormal)} tells us that 
	$ \bigl( (s^{p-1},t)^p\#, f \bigr)^{2p} $
is a hamiltonian cycle in $\Cay(G;S)$.
 \end{itemize}

\begin{subcase}
 Assume ${P}$ is nonabelian of
exponent~${p^2}$.
 \end{subcase}
 We have
 $$ P = \bigl\langle x,y \mid x^{p^2} = y^p = e, \text{$[x,y] = x^p$ is central}
\bigr\rangle . $$
Since $\langle [P,P], y \rangle$ is the unique elementary abelian subgroup of order~$p^2$ in~$G$, it must be normalized by~$f$. Thus, $\langle y \rangle$ must be in an eigenspace of the action of~$f$ on $P/[P,P]$, so we may assume $y^f \in \{y^{\pm 1}\}$. Also, by choosing $\langle x \rangle$ to also be in an eigenspace, we may assume $x^f \in \langle x \rangle$.

Since $\Aut \bigl( \langle x \rangle \bigr)$ is abelian, and $y$ acts nontrivially on~$\langle x \rangle$, we know that $y$ is not in the commutator subgroup of $\langle f, y \rangle$. So $f$ cannot invert~$y$. Therefore $f$ centralizes~$y$. Thus, $[G,G] \subset \langle x \rangle$, so $[G,G]$ is cyclic of prime-power order, so \cref{KeatingWitte} applies.

\begin{subcase}
 Assume ${P}$ is nonabelian of exponent~${p}$.
 \end{subcase}
  We have
 $$ P = \bigl\langle x,y, z \mid x^p = y^p = z^p = e, \text{$z = [x,y]$ is central} \bigr\rangle .$$
 We may assume $S \cap \langle z \rangle = \emptyset$, for otherwise \cref{NormalEasy} applies.

 \begin{subsubcase}
 Assume $|S| = 2$.
 \end{subsubcase}
 We have $S = \{a,s\}$.

 \begin{subsubsubcase}
 Assume $a^2 \neq e$.
 \end{subsubsubcase}
 We may assume $a^2 \notin [P,P]$ (otherwise \cref{DoubleEdge} applies), so there is no harm in assuming $a^2 = y^2$ (and $a$ obviously centralizes $\langle a^2 \rangle = \langle y\rangle$).
 
Note that, since $a \in fP$, the elements $a$ and~$f$ have the same action on $P/[P,P]$, and they have the same action on $Z(P)$. 
 
 Since $f$ acts as an automorphism of order~$2$ on $P/[P,P]$, and does not centralize $P$, it must act nontrivially on $P/[P,P]$, so $-1$ must be an eigenvalue of this action. Thus, we may assume $x^f \in x^{-1} [P,P]$. 
 Then, since $[x^{-1} , y] = [x,y]^{-1}$ (because $[x,y] \in Z(P)$), we have
 	\begin{align} \label{2p3pf-aInvertsP'}
	 z^f = [x,y]^f = [x^f, y^f] = [x^{-1} , y] = [x,y]^{-1} = z^{-1} ,
	 \end{align}
so $f$ inverts $\langle z \rangle$. Since $x^f \in x^{-1} \langle z \rangle$ (and $\langle x, z \rangle$ is abelian), this implies that 
	$$ \text{$f$ inverts $\langle x, z \rangle$.} $$
(So $x^f = x^{-1}$, $y^f = y$, and $z^f = z^{-1}$.)
Therefore, replacing $a$ by an appropriate conjugate, we may assume $a = fy$.

 \begin{itemize}
 \item If $s^a \in s^{-1}[P,P]$, then we may assume $s = x$. 
Also, since $[x,y] = z \in Z(P)$, we see that
	$$ x x^a = x x^{fy} = x (x^{-1})^y = [x^{-1},y] = [x,y]^{-1} = z^{-1} $$
generates $\langle z \rangle$. Hence, the path
	$$ \bigl( x^{-(p-1)}, a^{-1}, x^{-(p-1)},a \bigr)^{p}\# $$
visits all of the elements of $\langle x, z \rangle \cup a^{-1} \langle x, z \rangle$, so all of the vertices in this path are in different right cosets of $\langle y \rangle$. Then, since $a^{-2} = y^{-1}$ generates $\langle y \rangle$, \cref{FGL(notnormal)} tells us that 
	$$  \Bigl( \bigl( x^{-(p-1)}, a^{-1}, x^{-(p-1)},a \bigr)^{p}\#, a^{-1} \Bigr)^p $$
is a hamiltonian cycle in $\Cay(G;S)$.

 \item If $s^a \notin s^{-1}[P,P]$, then we may write $s = xy^\ell$ with $\ell \not\equiv 0 \pmod{p}$, and we may assume $\ell \not\equiv 1 \pmod{p}$, by replacing $a$ and~$y$ with their inverses if necessary.

Since $s \equiv x \pmod{\langle y, z \rangle}$, and $a$ inverts~$x \pmod{\langle y, z \rangle}$, it is clear that $(s^{p-1}, a)^2$ is a hamiltonian cycle in $\Cay \bigl( G/\langle y, z \rangle; S \bigr)$. Then, since the product $(s^{p-1}a)^2 = y^{2-2\ell}z^{\ell-1}$ generates $\langle y, z \rangle / \langle y \rangle$, \cref{FGL(skewgen)} tells us that $(s^{p-1}, a)^{2p}$ is a hamiltonian cycle in the quotient $\langle y \rangle \backslash \Cay( G; S)$. 
But 
	$$a^{-1} = a^{-2} a = y^{-2} a \in \langle y \rangle a, $$
so the final edge of this hamiltonian cycle is a multiple edge in the quotient. Thus, \cref{MultiDouble} provides a hamiltonian cycle in $\Cay(G;S)$.
\end{itemize}

 \begin{subsubsubcase}
 Assume $a^2 = e$.
 \end{subsubsubcase}
 We may assume $a = f$.
 Since $\langle s, s^f, [P,P] \rangle \normal G$, we have 
 	$$G = \langle f, s \rangle = \langle f \rangle \langle s, s^f, [P,P] \rangle ,$$
so $\{s,s^f\}$ must generate $P/[P,P]$. So $s^f \notin \langle s \rangle [P,P]$, which implies that the action of~$f$ on $P/[P,P]$ has two distinct eigenvalues (both $1$ and~$-1$), and that $s$ is not in either of these eigenspaces. Thus, we may assume $s = xy$ with $x^f = x^{-1}$ and $y^f = y$. Note that, from the calculation of \pref{2p3pf-aInvertsP'}, we also know $z^f = z^{-1}$. 

We claim that
	$$ \bigl( (s^{p-1}, f)^{2p-2}, (s^{-(p-1)}, f)^2 \bigr)^p $$
is a hamiltonian cycle in $\Cay(G;S)$. This walk is obviously of the correct length, and is closed (because $P$ has exponent~$p$), so we need only show that it visits all of the elements of~$G$.

We have
	\begin{align*}
	(s^{p-1}f)^{2p-2} (s^{-(p-1)}f)^2 &= (s^{-1}f)^{-2}(sf)^2 = y^4 , \\
	(s^{p-1}f)^2 &= (s^{-1}f)^2 = y^{-2} z , \\
	s &= xy 
	, \end{align*}
so the walk visits all vertices of the form
	$ (y^4)^i (y^{-2}z)^j(xy)^k$.
That is, it visits all of the vertices in~$P$.

Also, note that the first vertex of $fP$ visited is $s^{p-1} f$, and we have
	$$(s^{p-1}f)^{2p-3}(s^{-(p-1)}f)^2(s^{p-1}f) = (s^{-1}f)^{-3}(sf)^2(s^{-1}f) = y^4z^{-4} , $$
so the walk visits all vertices of the form
	\begin{align} \label{2p3pf-PcupfP}
	 (s^{p-1} f) \cdot (y^4z^{-4})^i (y^{-2}z)^j (xy)^k
	 \end{align}
 with $0 \le j \le p-2$.
In addition, since 
	\begin{align*}
	 (s^{p-1} f)^{2p-3}(s^{-(p-10}f) 
	 &= (s^{-1}f)^{-3}(sf) 
	 \\&= y^4 x^{-2} z^{-4} 
	\\& = (y^4 z^{-4})(y^{-2}z)^{-1}(xy)^{-2} 
	 \\&\in \langle y^4 z^{-4} \rangle (y^{-2}z)^{p-1} \langle xy \rangle
	 , \end{align*}
it also visits the vertices of the form \pref{2p3pf-PcupfP} with $j = p-1$. Thus, the walk visits all of the vertices in $fP$.

So the walk visits all of the vertices in $P \cup fP = G$, as claimed.

 \begin{subsubcase} \label{2p3-expp-S=3}
 Assume $|S| = 3$.
 \end{subsubcase}
Because $S$ is minimal, we must have $a^2 \in [P,P]$, so $\langle a^2 \rangle \normal G$.
So we may assume $a^2 = e$ (otherwise \cref{DoubleEdge} applies). Thus, we may assume $a = f$. We may also assume $s = x$, so we write $S = \{f,x,t\}$.

\begin{subsubsubcase}
Assume $t \in P$.
\end{subsubsubcase}
We may assume $S = \{f, x, y\}$.

\begin{itemize}
\item Suppose $x^f \in \{x^{\pm1}\}$.
From \cref{pk}, we know there is a hamiltonian cycle $(s_i)_{i=1}^{p^3}$ in $\Cay \bigl(P; \{x,y\} \bigr)$. We may assume $s_{p^3} = x^{-1}$. There is also a hamiltonian cycle $(t_i)_{i=1}^{p^3}$ in $\Cay \bigl( P; \{x,y\} \bigr)$, such that $t_{p^3} = x^f$. Then
	$$ \bigl( (s_i)_{i=1}^{p^3}\#, f, (t_i)_{i=1}^{p^3}\#, f \bigr) $$
is a hamiltonian cycle in $\Cay(G;S)$, because it traverses all the vertices in $P$, then all of the vertices in $fP$, and the final vertex is
	$$ (s_1s_2\cdots s_{p^3})(s_{p^3}^{-1} f) (t_1t_2\cdots t_{p^3})(t_{p^3}^{-1} f)
	= (e)(x^{-1} f)(e)(x^f f) 
	= e .$$

\item Suppose $x^f \notin \{x^{\pm1}\}$ and $y^f \notin \{y^{\pm1}\}$.
Because $S$ is minimal, we know that $x^f \in \langle x, [P,P] \rangle$ and $y^f \in \langle y, [P,P] \rangle$. Since $1$ cannot be the only eigenvalue of~$f$ on $P/[P,P]$, this means we may assume $x^f \in x^{-1}[P,P]$ (by interchanging $x$ and~$y$ if necessary).

We claim that $f$ inverts $P/[P,P]$. If not, then $f$ centralizes~$y \pmod{[P,P]}$, so, from the calculation of \pref{2p3pf-aInvertsP'}, we see that $f$ inverts $[P,P]$. This implies that $f$ does not centralize any element of $\langle x, z \rangle$, so it must invert all of these elements. This contradicts the assumption that $x^f \notin \{x^{\pm1}\}$.

The above claim implies that
	$\langle y x^{-1}, z \rangle \normal G$,
and that $(x^{p-1},f)^{2}$ is a hamiltonian cycle in 
	$ \Cay \bigl(G/ \langle y x^{-1}, z \rangle ; \{f,x\} \bigr) $.
Also note that $(x^{p-1}f)^2 \in [P,P]$, and we have
	$$ xyx^{p-3}fx^{p-1}f
	= \bigl( xyx^{-2}  \bigr)(x^{p-1}f)^2
	= \bigl( (yx^{-1}) [yx^{-1},x^{-1}]  \bigr)(x^{p-1}f)^2
	. $$
Thus, either $xyx^{p-3}fx^{p-1}f$ or $(x^{p-1}f)^2$ generates $\langle y x^{-1}, [P,P] \rangle / \langle yx^{-1} \rangle$, depending on whether $(x^{p-1}f)^2$ is trivial or not. Hence, \cref{FGL(skewgen)} tells us that either
	$$ \text{$\bigl( x,y,x^{p-3},f,x^{p-1},f \bigr)^p$ or $(x^{p-1},f)^{2p}$} $$
is a hamiltonian cycle in $\langle y x^{-1} \rangle \backslash \Cay(G;S)$. Then, since $y \in \langle yx^{-1} \rangle x$, the first edge of the hamiltonian cycle is doubled in the quotient multigraph, so \cref{MultiDouble} provides a hamiltonian cycle in $\Cay(G;S)$.
\end{itemize}

\begin{subsubsubcase}
Assume $t \notin P$.
\end{subsubsubcase}
This is very similar to the preceding argument.
Write $t = fy$, so $S = \{f, fy, x\}$. 

From the argument at the start of \cref{2p3-expp-S=3}, we see that we may assume $(fy)^2 = e$. Therefore $f$ inverts~$y$, so $\langle y, [P,P] \rangle \normal G$. 

Also, since the minimality of $S$ tells us $\langle f, x \rangle \neq G$, we know $x^f \in \langle x, [P,P] \rangle$, so there exists $\epsilon \in \{\pm 1\}$ such that $x^f \in x^\epsilon [P,P]$. Then it is easy to see that
$(f, x^{-(p-1)}, f, x^{\epsilon(p-1)})$ is a hamiltonian cycle in 
	$ \Cay \bigl( G / \langle y, z \rangle ; \{f,x\} \bigr) $.
For 
	$$z_1 = (f)(x^{-(p-1)})(f) (x^{\epsilon(p-1)}) = fxfx^{-\epsilon} \in [P,P] ,$$
we have
	$$ f(x^{-(p-1)})(fy) (x^{\epsilon(p-1)})
	= fx(fy) x^{-\epsilon}
	= fxfx^{-\epsilon} y [y,x^{-\epsilon}]
	= z_1 y [y,x^{-\epsilon}] 
	. $$
Thus, either $(f)(x^{-(p-1)})(fy) (x^{\epsilon(p-1)})$ or $(f)(x^{-(p-1)})(f) (x^{\epsilon(p-1)})$ generates $\langle y, [P,P] \rangle / \langle y \rangle$, depending on whether $z_1$ is trivial or not. Hence, either
	$$ \text{$\bigl( f, x^{-(p-1)}, fy, x^{\epsilon(p-1)} \bigr)^p$ or $\bigl( f, x^{-(p-1)}, f,  x^{\epsilon(p-1)} \bigr)^p$} $$
is a hamiltonian cycle in $\langle y \rangle \backslash \Cay \bigl( G; \{f, fy, x\} \bigr)$. Then, since $fy \in \langle y \rangle f$ (recall that $f$ inverts~$y$), the first edge of the hamiltonian cycle is doubled in the quotient multigraph, so \cref{MultiDouble} provides a hamiltonian cycle in $\Cay(G;S)$.
 \end{proof}


\section{Groups of order $18p$} \label{18pSect}

\begin{prop} \label{18p}
  If\/ $|G| = 18p$, where $p$ is prime, then every connected Cayley graph on~$G$ has
  a hamiltonian cycle.
\end{prop}

\begin{proof}
Let $S$ be a minimal generating set of~$G$.
  We may assume $p \ge 5$. (Otherwise either $|G| = 36 = 4 \cdot 3^2$, so \cref{4p2} applies, or  $|G| = 54 = 2 \cdot 3^3$, so \cref{2p3} applies.)

Note that $G$ is solvable (for example, this follows from the fact that
$|G| = 2 \times \mbox{odd}$, but can be proved quite easily), so  \cref{Hall-solv} tells us $G$ has
a Hall subgroup~$H_{18}$ of order~$18$, and also
has a Hall subgroup $H_{9p}$ of order~$9p$. Now
$H_{9p}$, being of index two, is normal in~$G$. Also, from Sylow's
Theorem \pref{NoMod1->Normal}, we see that the Sylow $p$-subgroup~$\integer_p$ is normal
(hence, characteristic) in $H_{9p}$. So $\integer_p$ is normal in~$G$.
Therefore $G = H_{18} \ltimes \integer_p$.

We may assume $H_{18}$ is nonabelian (otherwise \cref{KeatingWitte}
applies), so $H_{18}$ is either $D_{18}$, $\integer_2 \ltimes
(\integer_3
\times \integer_3)$ (dihedral type), or $D_6 \times Z_3$.

\setcounter{case}{0}

\begin{case}
  Assume ${H_{18} \iso D_{18}}$.
  \end{case}
  Then either $[G,G] = \integer_9$ (so \cref{KeatingWitte} applies) or $G
\iso D_{18p}$ (so \cref{dihedral} applies).

\begin{case}
  Assume ${H_{18} = \integer_2 \ltimes (\integer_3 \times \integer_3)}$ is
of dihedral type.
  \end{case}
  We may assume $G$ is not of dihedral type (otherwise
\cref{DihType}
applies), so $G = H_{18} \times \integer_p$. Let $s$ be an element
of~$S$ whose order is divisible by~$p$.

\begin{subcase}
Assume $|s| = p$.
\end{subcase}
Then $s \in \integer_p = Z(G)$, so \cref{NormalEasy} applies.

\begin{subcase}
Assume $|s| = 2p$.
\end{subcase}
Since $|G|$ is the product of only four primes, and $|s|$ is divisible by two of them, it is clear that $|S| \le 3$. On the other hand, it is clear that $G/ \integer_p \iso H_{18}$ has no $2$-element generating set (because two elements of order~$2$ always generate a dihedral group). Therefore $S$ is a minimal generating set of $G/ \integer_p$. Then, since $s$ and~$s^{-1}$ give a double edge in $\Cay(G/\integer_p;S)$, \cref{DoubleEdge} applies.

\begin{subcase}
Assume $|s| = 3p$.
\end{subcase}
  Let $f$ be a generator of $\integer_2$, $x$ and $y$ be generators
of $\integer_3 \times \integer_3$
  and $z$ be a generator of $\integer_p$. We may assume $S$
does not contain any elements of order~$p$ or~$2p$ (otherwise, a
preceding
case applies). We may also assume $S$ does not contain any elements of
order~$3$ (else \cref{NormalEasy} applies with $N = \integer_p$). Thus, each element
of~$S$ has order $2$ or~$3p$, so there are only two cases to
consider:

\begin{subsubcase}
Assume $S = \left\{ f, fx, yz \right\}$.
\end{subsubcase}
 Since $e$, $yz$ and $(yz)^2$ are in different right cosets of 
 $\langle (yz)^2 f \rangle = \langle fy, z \rangle$, \cref{FGL(notnormal)} tells us that
 	$$ \bigl( (yz)^{2}, f \bigr)^{2p} \#$$
 is a hamiltonian path in the subgraph induced by $\langle fy, z \rangle$. Therefore, all of the vertices of this path are in different right cosets of $\langle x \rangle$. So \cref{FGL(notnormal)} tells us that
 	$$ \Bigl( \bigl( (yz)^{2}, f \bigr)^{2p} \#, fx \Bigr)^3 $$
is a hamiltonian cycle in $\Cay(G;S)$.
 
  \begin{subsubcase}
Assume $S = \left\{ f, yz, xz^k \right\}$, with $k \not\equiv 0 \pmod{p}$.
\end{subsubcase}
  We may assume $k \not\equiv 3 \pmod{p}$ (by replacing $xz^k$ with its inverse, if necessary).
  Since $G/\langle xy^{-1}, z \rangle \iso D_6$, it is easy to see that
  	$$ \bigl( xz^k, yz, f, (yz)^2, f \bigr) $$
is a hamiltonian cycle in $\Cay \bigl( G/ \langle xy^{-1}, z \rangle; S \bigr)$.
Then, since 
	$$ (xz^k)(yz)(f)(yz)^2(f) = x y^{-1} z^{k + 3} $$
generates  $\langle xy^{-1}, z \rangle$, \cref{FGL} tells us that
  $$ \bigl( xz^k, yz, f, (yz)^2, f \bigr)^{3p}$$
  is a hamiltonian cycle in $\Cay(G ; S)$.

\begin{case}
  Assume ${H_{18} = D_6 \times \integer_3}$.
  \end{case}
We let $f$ and $x$ generate $D_6$, where $f^2 = x^3 = e$ and $x^f
= x^{-1}$. We let
$y$ generate $\integer_3$ and we let $z$ generate $\integer_p$.
Note that $y^f = y^x = y$, and that $z^x = z$ (since $x$ is in the commutator subgroup of~$H$).
  We may assume $H_{18}$ does not centralize $\integer_p$ (otherwise
\cref{KeatingWitte} applies). This implies that
$C_G(\integer_p)/[G,G]$ is a proper subgroup of $G/[G,G] \iso \integer_2
\times \integer_3$, so there are three possibilities for~$G$:
  \begin{itemize}
  \item $G = D_{6p} \times \integer_3$,
  or
  \item $G = D_6 \times (\integer_3 \ltimes \integer_p)$,
  or
  \item $G = (D_6 \times \integer_3) \ltimes \integer_p)$, where $D_6$
and $\integer_3$ both act nontrivially on~$\integer_p$.
  \end{itemize}
  In each case, since $Z(G) \cap \langle x , z \rangle = \{e\}$, we may assume
	$$S \cap \langle x , z \rangle = \emptyset$$
(else \cref{NormalEasy} applies).

\begin{subcase}
  Assume ${G = D_{6p} \times \integer_3}$.
  \end{subcase}
This implies $z^f = z^{-1}$ and $z^y = z$.
There exists an element of the form $fx^iy^jz^k$ in $S$.

\begin{subsubcase}
Assume there exists an element of the form $fx^iyz^k$ in $S$.
\end{subsubcase}
Conjugating by powers of $x$ and~$z$, we can assume $fy \in S$.
Now $(fy)^2 = y^2$
and thus, since $\langle y \rangle \normal G$, either
\cref{DoubleEdge} applies or
$S - \{fy\}$ generates $G / \langle y \rangle$. Assume
the latter. This clearly implies
that another element of the form $fx^iy^jz^k$ is in $S$ and that $|S| \ge
3$. 
Since the index of $\langle f y \rangle$ is $3p$, which has only two prime factors,
we conclude that $|S| = 3$. Thus, the minimality of~$S$ implies that precisely one of
$i$ and $k$ must be zero. 

Since the minimality of~$S$ implies no element of the form
$xy^{j'}z$ is in $S$, and since neither $\{fxy^{j'},
y^{j''}z\}$ nor $\{fy^{j'}z, xy^{j''}\}$ generates $G /
\langle y \rangle$, it follows
that $S = \{fy, fxy^{j}, fy^{j'}z\}$. But since $S$ is
minimal, we must have $j = j' = 0$. Thus
$S = \{fy, fx, fz\}$.
Now
	\begin{itemize}
	\item $(fz,fx)^{3p}\#$ is a hamiltonian path in $\Cay \bigl( \langle f,x,z \rangle ; \{fz, fx\} \bigr)$, so all the vertices in this path are obviously in different right cosets of $\langle x^{-1} y \rangle$,
	and
	\item $\bigl( (fz) (fx) \bigr)^{3p} (fx)^{-1} (fy) = x^{-1} y$ obviously generates $\langle x^{-1} y \rangle$,
	\end{itemize}
so \cref{FGL(notnormal)} implies that 
	$$ \bigl( (fz, fx)^{3p} \#, fy \bigr) ^3 $$
is a hamiltonian cycle in $\Cay(G;S)$. 

\begin{subsubcase}
Assume that $S$ does not contain any element of the form $fx^iy^jz^k$ with $j \not\equiv 0 \pmod{3}$.
\end{subsubcase}
Then we can assume
$f \in S$. There must be an element of  the form $x^iyz^k$ in $S$. Note
that we can assume
that at least one of $i$ and~$k$ is nonzero for otherwise
\cref{NormalEasy} applies.

\begin{subsubsubcase}
Assume $i$ and $k$ are both nonzero. 
\end{subsubsubcase}
Then we can assume $S = \{f, xyz\}$.
Since $(xyz)^{3p-1}$ is a hamiltonian path in $\Cay \bigl( \langle xy, z \rangle ; \{xyz\} \bigr)$, it is clear that all of the vertices in this path are in different right cosets of $\langle fxz , y \rangle = \langle (xyz)^{3p-1} f \rangle$. So \cref{FGL} tells us that
$ \bigl( (xyz)^{3p-1} , f \bigr)^6 $ is a hamiltonian cycle in $\Cay(G;S)$.

\begin{subsubsubcase}
Assume $i \neq 0$ and $k = 0$.
\end{subsubsubcase}
 We can assume $xy
\in S$. Then, since $S \cap \langle x,z \rangle = \emptyset$,  the only candidates
for the third element of~$S$ are $yz$ and $fx^iz$.
	\begin{itemize}
	\item Suppose $yz \in S$. Note that every element of the abelian group $\langle x, y, z \rangle$ can be written uniquely in the form $(yz)^i (xy)^{-j}$, where $0 \le i < 3p$ and $0 \le j < 3$, so it is easy to see that
	$$ \bigl( (yz)^{3p-1}, (xy)^{-1}, (yz)^{-(3p-1)}, (xy)^{-1}, (yz)^{3p-1} \bigr) $$
is a hamiltonian path in $\Cay \bigl( \langle x, y, z \rangle; \{xy, yz\} \bigr)$. Thus, letting
	$$  g = (yz)^{3p-1} (xy)^{-1} (yz)^{-(3p-1)} (xy)^{-1} (yz)^{3p-1} f
	= x z^{-1} f ,$$
it is clear that all of the vertices of this path are in different right cosets of $\langle g \rangle$ (since $|g| = 2$). Therefore \cref{FGL(notnormal)} tells us that
	$$ \bigl( (yz)^{3p-1}, (xy)^{-1}, (yz)^{-(3p-1)}, (xy)^{-1}, (yz)^{3p-1}, f \bigr)^2 $$
is a hamiltonian cycle in $\Cay(G;S)$.

	\item Suppose $fx^iz \in S$. We may assume $i \not\equiv 0 \pmod{3}$, for otherwise $f \equiv fx^iz \pmod{\langle z \rangle}$, so \cref{DoubleEdge} applies with $N = \langle z \rangle$. Then we may assume $i = 1$ (by replacing $x$ and~$xy$ by their inverses, if necessary). So $S = \{ f, fxz, xy \}$. Now, $(xy)^2 f (xy)^{-2}
fxz = x^2 z$, which generates
the normal cyclic subgroup $\langle x,z \rangle = \langle xz\rangle$. Then, since
	$$1, xy, x^2y^2, fxy^2, fy, fx^2$$
lie in different cosets of the subgroup $\langle xz \rangle$, \cref{FGL} implies that
$((xy)^2,f,(xy)^{-2},fxz)^{3p}$ is a hamiltonian cycle in $\Cay(G;S)$.
	\end{itemize}

\begin{subsubsubcase} \label{SubSubSubi=0andknot0}
Assume $i = 0$ and $k \neq 0$. 
\end{subsubsubcase}
This means $yz \in S$.
We may assume the third element of~$S$ does not belong to $\langle x, y, z \rangle$ (otherwise a previous \lcnamecref{SubSubSubi=0andknot0} applies, since $S \cap \langle x, z \rangle = \emptyset$).
Then the third element of~$S$ must be of the form $fxz^k$.
It is easy to see that 
$$ \text{$\bigl( (yz)^2, f, (yz)^{-2}, fxz^k \bigr)$ 
\quad and \quad
$\bigl( (yz)^{-2}, f, (yz)^2, fxz^k\bigr)$} $$
are hamiltonian cycles in $\Cay \bigl( G/ \langle x,z \rangle; S \bigr)$. Since one or the other of
	$$ \text{$(yz)^2 f (yz)^{-2} (fxz^k) = xz^{k+4}$
	\quad and \quad
	$(yz)^{-2} f (yz)^{2} (fxz^k) = xz^{k-4}$} $$
generates $\langle x,z \rangle$ , we see from \cref{FGL} that either
$$ \text{$\bigl( (yz)^2, f, (yz)^{-2}, fxz^k \bigr)^{3p}$ 
\quad or \quad
$\bigl( (yz)^{-2}, f, (yz)^2, fxz^k\bigr)^{3p}$} $$
is a hamiltonian cycle in $\Cay(G;S)$.

\begin{subcase}
  Assume ${G = D_6 \times (\integer_3 \ltimes \integer_p)}$.
  \end{subcase}
Note that this implies $z^f = z$ and $z^y =
z^r$, where
$r^3 \equiv 1\pmod{p}$ and $r \neq 1$.
  (We must have $p \equiv 1 \pmod{3}$.)

Suppose there exists $s \in S$ whose projection to the second factor is
 a
nontrivial element of~$\integer_p$. We may assume the first component is
a reflection (otherwise it generates a normal subgroup $\langle x^i z
\rangle$
which clearly has a trivial intersection with the center of $G$, so
\cref{NormalEasy}
applies). That is, $s = fx^iz$. Clearly we can assume $s = fz \in
S$ (conjugate by a power of~$x$).
Then $s$ yields a double edge in $G/\integer_p$, so, by \cref{DoubleEdge}, we may
assume $S - \{s\}$ generates $G/\integer_p$. From the minimality of~$S$, we know $\langle S - \{s\} \rangle \neq G$, so we conclude that $\langle S - \{s\} \rangle = D_6 \times \integer_3$ (or a conjugate). Furthermore, since $|G|$ is the product of only four primes, and $|s|$ is divisible by two of them, we know $|S| \le 3$. Therefore, some element~$t$ of $S - \{s_1\}$ must project nontrivially to both $D_6$ and $\langle y\rangle$. Since $\langle s , t \rangle \neq G$, the projection of $t$ to~$D_6$ must be~$f$, so we may assume $t = fy$. Then the final element of $S$ must be of the form $f x^i$, with $i \not\equiv 0 \pmod{3}$. Therefore, we may assume $S = \{fz, fy, fx\}$.
In this case, \cref{Stud71} applies with $s_1 = fy$ and $s_2
= fz$, because $s_1 s_2 = yz$ has order~$3$, and $\langle S - \{s_1\} \rangle = \langle f, x, z \rangle$ has order~$6p$.

We may now assume that
	\begin{align} \label{ProjNotZp}
	\text{the projection of a generator to the second
factor is never a nontrivial element of~$\integer_p$.} 
	\end{align}

\begin{subsubcase}
  Assume ${|S| = 2}$.
  \end{subsubcase}
The generating set of $\integer_3 \ltimes \integer_p$ must be of the form $\{y, yz\}$, and the 
generating set of~$D_6$ is either two reflections or a rotation and a reflection. We now discuss each of the possibilities individually:

\begin{subsubsubcase}
  Assume $S = \{fy, fxyz\}$. 
  \end{subsubsubcase}
  Since $fy$ generates the cyclic group $G/ \langle x z \rangle$, it is obvious that $ \bigl( (fy)^5, fxyz \bigr)$ is a hamiltonian cycle in $\Cay \bigl( G/ \langle xz \rangle ; S \bigr)$. Then, since
  	$ (fy)^5 (fxyz) = xz $,
\cref{FGL} tells us that $\bigl( (fy)^5, fxyz \bigr)^{3p}$ is a hamiltonian cycle in $\Cay(G;S)$.

\begin{subsubsubcase}
  Assume $S = \{fy, xyz\}$. 
    \end{subsubsubcase}
Much as in the previous paragraph, it is easy to see that
	 $$ \bigl( (xyz)^2, (fy)^{-1}, (xyz)^{-2}, fy \bigr) $$
is a hamiltonian cycle in $\Cay \bigl( G/ \langle xz \rangle ; S \bigr)$. Therefore, since
	$$ (xyz)^2 (fy)^{-1} (xyz)^{-2}(fy)
	= x^4 z^{r-1} $$
generates $\langle xz \rangle$, \cref{FGL} tells us that
	$$ \bigl( (xyz)^2, (fy)^{-1}, (xyz)^{-2}, fy\bigr)^{3p }$$
is a hamiltonian cycle in $\Cay(G;S)$.

\begin{subsubcase}
  Assume ${|S| = 3}$.
  \end{subsubcase}

\begin{subsubsubcase}
  Assume ${S \cap \bigl( D_6 \times \integer_p \bigr) \neq \emptyset}$.
  \end{subsubsubcase}
  Then (from \pref{ProjNotZp} and the fact that $S \cap \langle x ,z \rangle = \emptyset$) we may assume $f \in S$. There must be an element
whose projection to both $D_6$ and $\integer_3 \ltimes \integer_p$ is
nontrivial. Since by assumption $S$ does not contain any element of the form
$f^\ell x^iz^k$ with $k \neq 0$, we are left with two possibilities.
  \begin{itemize}

  \item Assume $fxy \in S$. Because $S$ generates~$G$, the third
element of~$S$ must be of the form $f^\ell x^iyz$ (or its inverse).
Since $S$ is
minimal, this element must either be $yz$ or $fxyz$. Thus,
$S$ is either $\{f, fxy, yz \}$ or $\{f, fxy, fxyz \}$.
In either case, taking
$s_1 = f$ and $s_2 = fxy$ we get that $s_1s_2 = xy$ is of order
$3$ and $\langle S - \{s_1\}\rangle = \langle fx, y, z \rangle$ 
 is of order $6p$, so
clearly \cref{Stud71} applies.

  \item Assume $xy \in S$.  Since $S$ generates~$G$, the third
element
of~$S$ must be of the form $f^\ell x^iyz$ (or its inverse).
Since $S$ is minimal, we must have $\ell = 0$. There
are three Cayley graphs to consider:
  \begin{itemize}
  \item Suppose $S = \{f, xy, yz\}$. Taking
$s_1  = (xy)^{-1} = x^2y^2$ and $s_2 = yz$, we see that
  $s_1s_2 = x^2z$ is of order $3p$ and clearly $|\langle S
- \{s_1^{-1}\}\rangle| = 6$. So \cref{Stud71} applies.
  \item Suppose $S =  \{f, xy, xyz \}$.
Taking $s_1 = (xy)^{-1} = x^2y^2$ and $s_2 = xyz$, we see that $s_1s_2 =
z$ is of order $p$, and  $\langle S - \{s_1^{-1}\}\rangle  = \langle f, x, yz \rangle$ has order~$18$. So
\cref{Stud71} applies.
  \item Suppose $S = \{f, xy, x^2yz\}$. 
 Since $G/ \langle x,z \rangle$ is abelian, it is easy to see that
 	$$ \bigl(f, (xy)^{-2}, f, xy, x^2yz \bigr) $$
is a hamiltonian cycle in $\Cay \bigl( G/ \langle x,z \rangle ; S \bigr)$.
Therefore, since
	$$ (f) (xy)^{-2} (f) (xy) (x^2yz) = x^2 z $$
generates $\langle xz \rangle$, \cref{FGL} tells us that
	$$\bigl(f, (xy)^{-2}, f, xy, x^2yz \bigr)^{3p} $$
is a hamiltonian cycle in $\Cay(G;S)$.
  \end{itemize}
  \end{itemize}

\begin{subsubsubcase}
  Assume ${S \cap \bigl( D_6 \times \integer_p \bigr) = \emptyset}$.
  \end{subsubsubcase}
  We must
have an element of the form $fx^iyz^k$ in $S$, so we can assume 
	$$fy \in S .$$
In order to generate $D_6$, the set
$S$ must also contain an element of the form $fxy^jz^k$ or
$xy^jz^k$.  Furthermore, the assumption of this paragraph implies $j \neq 0$, so we may assume $j = 1$, by passing to the inverse if necessary.

  \begin{itemize}

  \item Suppose $fxyz^k \in S$. Since $S$ is minimal and $|S| =
3$, we must have $k=0$; that is, $fxy
\in S$.
Then, because $S$ generates~$G$, the third element
of~$S$ must be of the form $f^\ell x^iyz$. If $i \neq 0$, then
$\langle fy, f^\ell x^iyz \rangle = G$, while if $i = 0$ and $\ell = 1$,
then $\langle fxy, fyz\rangle = G$. These conclusions contradict the minimality of~$S$,
 so there is only one Cayley graph to consider: we have $S = \{fy, fxy, yz\}$. 
 Taking $s_1 = fxy$ and $s_2 = (fy)^{-1} = fy^2$, we get
that $s_1s_2 = x^2$ is of order $3$, and $\langle S -
\{s_1\}\rangle = \langle f, y, z \rangle$ is clearly of order $6p$. So \cref{Stud71}
applies.

  \item Suppose $xyz^k \in S$. Since $S$ is minimal and $|S| = 3$
we must have $k = 0$, so $xy \in S$. Then, because $S$ generates~$G$, the third
element
of~$S$ must be of the form $f^\ell x^iyz$. If $i \neq 0$, then
$\langle fy, f^\ell x^iyz \rangle = G$, while if $i = 0$ and $\ell = 1$,
then $\langle xy, fyz\rangle = G$. So there is only
one Cayley graph to consider:
  we have $S = \{fy, xy, yz\}$.
Taking $s_1 = fy$ and $s_2 = (xy)^{-1} = x^2y^2$, we get that
$s_1s_2 = fx^2$ is of order $2$ and $\langle S
- \{s_1\}\rangle = \langle x, y, z \rangle$ is of order~$9p$. So \cref{Stud71} applies.
   \end{itemize}

\begin{subsubcase}
  Assume ${|S| = 4}$.
  \end{subsubcase}
 Since $|G| = 18p$ is the product of only four prime factors, the order of the subgroup generated by any two elements of~$S$ must be the product of only two prime factors. It is easy to see that this implies every element of~$S$ belongs to either $D_6 \times \{e\}$ or $\{e\} \times (\integer_3 \ltimes \integer_p)$. Therefore, $\Cay(G;S)$ is isomorphic to
  $$\Cay(D_6;S_1) \times \Cay(\integer_3 \ltimes \integer_p;S_2) .$$
  Since the Cartesian product of hamiltonian graphs is hamiltonian, we conclude that $\Cay(G;S)$ has a hamiltonian cycle.

\begin{subcase}
  Assume ${G = (D_6 \times \integer_3) \ltimes \integer_p}$, where $D_6$ and $\integer_3$ both act nontrivially on~$\integer_p$.
  \end{subcase}
  (Note that we must have $p \equiv 1 \pmod{3}$.)
  This implies $z^f = z^{-1}$ and $z^y =
z^{r}$ where $r^3 \equiv 1 \pmod{p}$ (but $r \not\equiv 1 \pmod{p}$).

\begin{subsubcase} \label{D6xZ3xZpPf-S=2}
  Assume ${|S| = 2}$.
  \end{subsubcase}

\begin{subsubsubcase}
    Assume $S \cap \langle f, x, z \rangle = \emptyset$.
  \end{subsubsubcase}
The generating set~$S$ must contain an element of the form $f x^i y^j z^k$. By assumption, we must have $j \neq 0$, so we may assume $j = 1$. Then, conjugating by an element of $\langle x, z \rangle$, we may assume $fy \in S$.

To generate $G$, the second element of~$S$ must be of the form $f^\ell x y^{j'} z$. By assumption, we must have $j' \neq 0$, so we may assume $j' = 1$.  Therefore,  there are only two possibilities, and we discuss each of them individually:

  \begin{itemize}
  \item Suppose $S = \{fy, xyz\}$. Since
  $\bigl( (fy)^{-1}, (xyz)^{-2}, fy, (xyz)^2\bigr)$ is a hamiltonian cycle in $\Cay \bigl( G/\langle x, z \rangle ; S \bigr)$, and
  	$$ \text{$(fy)^{-1}(xyz)^{-2}(fy)(xyz)^2 = x z^{(r+1)^2}$ generates $\langle x, z \rangle$} , $$
\cref{FGL} tells us $\bigl( (fy)^{-1}, (xyz)^{-2}, fy, (xyz)^2\bigr)^{3p}$ is a
hamiltonian cycle in $\Cay(G;S)$.
 
   \item Suppose $S = \{fy, fxyz\}$. 
   Since $fxyz \equiv fy \pmod{\langle x,z \rangle}$, it is obvious that $\bigl( (fy)^5, fxyz \bigr)$ is a hamiltonian cycle in $\Cay \bigl( G/\langle x, z \rangle ; S \bigr)$. Then, since $(fy)^5(fxyz) = xz$ generates $\langle x, z \rangle$, \cref{FGL} tells us that
   $\bigl( (fy)^5, fxyz \bigr)^{3p}$ is a hamiltonian cycle in $\Cay(G;S)$.
  \end{itemize}

\begin{subsubsubcase}
  Assume $S \cap \langle f, x, z \rangle \neq \emptyset$.
  \end{subsubsubcase}
  Since $S \cap \langle x, z \rangle = \emptyset$, we must have $S \cap f \langle x, z \rangle \neq \emptyset$. Then, conjugating by an element of $\langle x, z \rangle$, we may assume $f \in S$.
To generate $G$, the second element of~$S$ must be of the form $f^\ell x y z$.

  \begin{itemize}
  \item Suppose $S =  \{f, xyz\}$. Since
  $\bigl( f, (xyz)^{-2}, f, (xyz)^2 \bigr)$ is a hamiltonian cycle in $\Cay \bigl( G/\langle x, z \rangle ; S \bigr)$, and
  	$$ \text{$f (xyz)^{-2} f (xyz)^2 = x z^{2(r+1)}$ generates $\langle x, z \rangle$} , $$
\cref{FGL} tells us that $\bigl( f, (xyz)^{-2}, f, (xyz)^2 \bigr)^{3p}$ is a
hamiltonian cycle in $\Cay(G;S)$.

  \item Suppose $S =  \{f, fxyz\}$.
  We may assume $4r \not\equiv -5 \pmod{p}$ (by replacing $y$ with its inverse, if necessary). 
  Let 
 	$$(s_i)_{i=1}^{18} 
	= \bigl( fxyz, f, (fxyz)^{-2}, f, (fxyz)^{-3}, f, (fxyz)^3, f, (fxyz)^2, f, (fxyz)^{-1}, f \bigr) .$$
Using the fact that $r^2 + r + 1 \equiv 0 \pmod{p}$, we calculate that the vertices of this walk are:
	\begin{align*}
	 e, \ 
	 &fxy z, \ 
	  x^2y z^{-1}, \ 
	  f x^2 z^{-2r-2}, \ 
	  x^2 y^2 z^{-3r-1}, \ 
	  fxy^2 z^{3r+1}, \ 
	  y z^{-3}, \ 
	  \\& fx z^{-4r-4}, \ 
	  y^2 z^{-5r-1}, \ 
	  fy^2 z^{5r+1}, \ 
	 x z^{4r+6}, \ 
	  fy z^{5-2r}, \ 
	  xy^2 z^{-7r-1}, \ 
	  \\& fx^2y^2 z^{7r+1}, \ 
	  x^2 z^{6r+8}, \ 
	  fx^2 y z^{7-2r}, \ 
	  xy z^{2r-7}, \ 
	  f z^{-8r-10}, \ 
	   z^{8r+10} 
	 . \end{align*}
Then, by modding out $\langle z \rangle$, we see that this walk visits the vertices of $G/\integer_p \iso D_6 \times \integer_3$ in the order
	$$ e, fxy, x^2y, f x^2 , x^2 y^2, fxy^2, y, fx, y^2, fy^2, x, fy, xy^2, fx^2y^2, x^2, fx^2 y, xy, f, e ,$$
so it is a hamiltonian cycle in 
  $\Cay( G/\integer_p ; S)$. 
  Furthermore, from our assumption that $4r \not\equiv -5 \pmod{p}$, we see that the final vertex $z^{8r+10}$ is not trivial in~$G$, so it generates $\langle z \rangle$.
Therefore \cref{FGL} provides a hamiltonian cycle in $\Cay(G;S)$.
  \end{itemize}

\begin{subsubcase}
  Assume ${|S| = 3}$.
  \end{subsubcase}

\begin{subsubsubcase}
  Assume some element of~${S}$ has order~${6}$.
  \end{subsubsubcase}
  Then $S$ contains $fy$ (or a conjugate). The only proper subgroups
of~$G$ that properly contain $fy$ are $\langle f \rangle \times (\integer_3
\ltimes \integer_p)$ and $D_6 \times \langle y \rangle$. Thus, recalling
the assumption that $S \cap \langle x, z \rangle = \emptyset$:
  \begin{itemize}
  \item the second generator can be assumed to be $yz$,
$fz$, or $fyz$,
  and
  \item the third generator can be assumed to be $fx$, $fxy$, or
$xy$.
  \end{itemize}

  We consider each possible choice of the second generator.

  \begin{enumerate} \renewcommand{\theenumi}{\alph{enumi}}
  \item If $yz \in S$, then any of the possible third generators can
be used:
  \begin{itemize}
  \item Suppose $S = \{fy, yz, fx\}$.
Taking $s_1 = fy$ and $s_2 = fx$ we get that $s_1s_2 = xy$ is
of order $3$ and since $\langle S - \{s_1\}\rangle = \langle fx, y, z \rangle$ has order $6p$, it is easy to see that \cref{Stud71} applies.
  \item Suppose $S = \{fy, yz, fxy\} $. 
  Taking $s_1 = fy$ and $s_2 = fxy$ we get that $s_1s_2 =
xy^2$ is of order $3$ and since $\langle S -
\{s_1\}\rangle = \langle f x, y, z \rangle$ has order~$6p$, it is easy to see that \cref{Stud71} applies.
  \item Suppose $S = \{fy, yz, xy\}$.
Taking $s_1 = fy$ and $s_2 = (xy)^{-1}$ we get that $s_1s_2 =
fx^2$ is of order $2$ and since $\langle S -
\{s_1\}\rangle = \langle x, y, z \rangle$ has order~$9p$, it is easy to see that  \cref{Stud71} applies.
  \end{itemize}

  \item If $fz \in S$, then, because $\langle fz, fxy \rangle =
G = \langle fz, xy \rangle$, there is only one possibility for the third
generator: we have $S = \{fy, fz, fx\}$.
Taking $s_1 = fx$ and $s_2 = fy$ we get that $s_1s_2 = x^2y$ is
of order $3$ and since $\langle S -
\{s_1\}\rangle = \langle f, y, z \rangle$ has order~$6p$, it is easy to see that \cref{Stud71} applies.

  \item If $fyz \in S$, then, because
   $\langle f y z, f x \rangle$, $\langle f y z, f x y \rangle$, and $\langle f y z, x y \rangle$ are all equal to~$G$, 
   none of the possible third generators yield a minimal generating set of~$G$. So there are no Cayley graphs to consider in this case.
  \end{enumerate}

\begin{subsubsubcase}
  Assume no element of~${S}$ has order~${6}$.
  \end{subsubsubcase}
  Then $S$ contains $f$ (or a conjugate). There must be an element
of~$S$ that does not belong to $D_6 \times \integer_p$ (that is, an
element of the form $f^\ell x^i y z^k$). Because there is no
element of order~$6$, we must have $\ell = 0$, so
the possibilities are:
  $y$, $yz$, $xy$, and $xyz$. However, we eliminate the
last option, because $\langle f, xyz \rangle = G$.

We consider each of the remaining possibilities:
  \begin{enumerate} \renewcommand{\theenumi}{\alph{enumi}}

  \item Suppose $y \in S$. The third generator must involve both $x$
and~$z$. Since $\langle f, xyz \rangle = G$ and since we
assumed $S \cap \langle x, z \rangle = \emptyset$, there is only one
possibility, namely, $S = \{f, y, fxz\}$. 
Taking $s_1 = fxz$ and $s_2 = f$ we get that $s_1s_2 =
x^2z^{-1}$ is of order $3p$ and since $\langle S
- \{s_1\}\rangle = \langle f, y \rangle$ has order~$6$, it is easy to see that \cref{Stud71} applies.

  \item Suppose $yz \in S$. The third generator must involve~$x$.
Since $S$ does not contain an element of order~$6$, or any element of $\langle x, z \rangle$, and $\langle f, xyz^k \rangle = G$ for $k \neq 0$,
 the only possibilities are $f x z^k$ and $x y$.
 
  \begin{itemize}
  \item Suppose $S =  \{f, yz, fxz^k\}$. Note that, because
  	$$ 2(r^2 + r) + 2(r+1) = 2(r+1)^2 \not\equiv 0 \pmod{p} ,$$
it cannot be the case that $k + 2(r^2 + r)$ and $k - 2(r+1)$ are both~$0$ modulo~$p$.
Therefore, $\langle x, z \rangle$ is generated by either 
  	$$\text{$(yz)^2(f) (yz)^{-2} (fxz^k) = x z^{k + 2(r^2 + r)}$
	or $(yz)^{-2} (f) (yz)^2 (fxz^k)  = x z^{k - 2(r +1)}$}  , $$
so \cref{FGL} tells us that either
 $$ \text{$\bigl( (yz)^2, f, (yz)^{-2}, fxz^k \bigr)^{3p}$ or
  $ \bigl( (yz)^{-2}, f, (yz)^2, fxz^k \bigr)^{3p}$} $$
 is a hamiltonian cycle in $\Cay(G;S)$.

  \item Suppose $S = \{f, yz, xy \}$. Since 
  	$ (xy)^{-2} f (yz)^2 f = x z^{-(r+1)} $
generates $\langle x, z \rangle$, \cref{FGL} tells us that
  $$ \bigl( (xy)^{-2}, f, (yz)^2, f \bigr)^{3p}$$ is a
hamiltonian cycle in $\Cay(G;S)$.
  \end{itemize}

  \item Suppose $xy \in S$. The third generator must involve~$z$.
  However, $\langle xy, fx^i y^j z \rangle = G$, and $\langle f, x^iyz \rangle$ is also equal to~$G$ if $i \neq 0$. 
  Since there is no element of the
form $x^iz^k$ in $S$, this implies that the
only possibility for the third generator is $yz$, so $S = \{f,
xy, yz\}$, but this generating set was already considered in the 
preceding paragraph. 
\end{enumerate}

\begin{subsubcase}
  Assume ${|S| = 4}$.
  \end{subsubcase}
  Some 3-element subset~$S'$ of~$S$ must generate $G/\integer_p$. Then,
because $S$ is minimal, we must have $\langle S' \rangle = D_6 \times
\integer_3$ (or a conjugate). Since $S'$ must be minimal, and $S' \cap
\langle x \rangle = \emptyset$, we must
have $S' = \{ f, fx, y \}$. 

Now the final element of~$S$ must be of the form $f^\ell x^i y^j z$. Since $S \cap \langle x,z \rangle = \emptyset$, we know that $\ell$ and~$j$ cannot both be~$0$.
	\begin{itemize}
	\item If $\ell \neq 0$, then either $\langle fx, y,  f^\ell x^i y^j z \rangle = G$, 
	or  $\langle f, y,  f^\ell x^i y^j z \rangle = G$, depending on whether $i$~is~$0$ or not.
	\item If $j \neq 0$, then $\langle f, fx, f^\ell x^i y^j z \rangle = G$.
	\end{itemize}
These conclusions contradict the minimality of~$S$, so there are no Cayley graphs to consider in this case.
\end{proof}

\begin{ack}
We are grateful to D.\,Jungreis, E.\,Friedman, and J.\,A.\,Gallian for making the unpublished manuscript \cite{JungreisFriedman} available to us. We also thank two anonymous referees for their helpful comments.

D.\,W.\,M.\ was partially supported by a research grant from the Natural Sciences and Engineering Research Council of Canada.
K.\,K., D.\,M., and P.\,\v S.\ were partially supported by Agencija za raziskovalno dejavnost Republike Slovenije, research program P1-0285.
\end{ack}


\end{document}